\documentclass[12pt,twoside,leqno]{article}
\usepackage{amsmath,amsthm,amssymb}
\numberwithin{equation}{section}
\usepackage[dvips]{graphicx,color,psfrag}

\makeatletter
\newcommand{\figcaption}[1]{\def\@captype{figure}\caption{#1}}
\newcommand{\tblcaption}[1]{\def\@captype{table}\caption{#1}}
\makeatother

\makeatletter

\ifcase\@ptsize
\def\rpkern{\mathchoice{\kern-1.45em}{\kern-1.11em}{}{}}%
\def\grpkern{\mathchoice{\kern-1.013em}{\kern-0.825em}{}{}}%
\or
\def\rpkern{\mathchoice{\kern-1.44em}{\kern-1.11em}{}{}}%
\def\grpkern{\mathchoice{\kern-1.00em}{\kern-0.81em}{}{}}%
\or
\def\rpkern{\mathchoice{\kern-1.472em}{\kern-1.14em}{}{}}%
\def\grpkern{\mathchoice{\kern-1.00em}{\kern-0.815em}{}{}}%
\fi

\def\minibullet{\mathchoice%
{\raise0.2ex\hbox{$\scriptstyle\bullet$}}%
{\raise0.26ex\hbox{$\scriptscriptstyle\bullet$}}{}{}}
\def\butabullet{\mathchoice%
{\raise0.8ex\hbox{$\scriptstyle\bullet$}{\kern-0.365em}%
\lower0.4ex\hbox{$\scriptstyle\bullet$}}%
{\raise0.75ex\hbox{$\scriptscriptstyle\bullet$}{\kern-0.335em}%
\lower0.25ex\hbox{$\scriptscriptstyle\bullet$}}{}{}}

\def\customprod#1#2%
{\operatorname{\coprod\kern#2{#1}\kern#2\prod}\displaylimits}

\renewcommand{\Re}{\mathrm{Re}\,}
\renewcommand{\Im}{\mathrm{Im}\,}

\theoremstyle{theorem}
\newtheorem*{multitheorem}{\variable@name}

\theoremstyle{definition}
\newcommand{\variable@name}{Theorem}
\newtheorem*{multiproclaim}{\variable@name}

\theoremstyle{plain}
\newtheorem{thm}{Theorem}[section]
\newtheorem{prop}[thm]{Proposition}
\newtheorem{lem}[thm]{Lemma}
\newtheorem{cor}[thm]{Corollary}

\theoremstyle{definition}
\newtheorem{dfn}[thm]{Definition}
\newtheorem{chu}[thm]{Remark}
\newtheorem{exa}[thm]{Example}

\title{Bilateral zeta functions and their applications}
\author{GENKI SHIBUKAWA}
\date{\empty}
\begin{document}
\maketitle

\begin{abstract}
We introduce a new type of multiple zeta functions, which we call bilateral zeta functions.  
We prove that the bilateral zeta function has a nice Fourier series expansion and the Barnes zeta function can be expressed as a finite sum of bilateral zeta functions. 
By these properties of the bilateral zeta functions, we obtain simple proofs of some formulas, for example, the reflection formula for the multiple gamma function, the inversion formula for the Dedekind $\eta$-function, Ramanujan's formula, Fourier expansion of the Barnes zeta function and multiple Iseki's formula. 
\end{abstract}
\section{Introduction}
In the present paper, we introduce a new type of multiple zeta functions(series) analogous to 
the ones Barnes originally defined in 1904 [Ba], and investigate their fundamental properties.
The Barnes zeta functions have been used for the proofs of the basic facts like 
the transformation formula for theta functions and
the reflection formula for multiple gamma functions.
Those known proofs are, however,  not clear enough, mainly because good explicit expressions
of the Barnes zeta functions have not been established.  A closer look at this situation tells us that  those proofs seem to have
paid extra efforts for pursuing  the parallelism to the most primitive case.
To these difficulties, we will demonstrate how our new type of multiple zeta functions, which we call {\it bilateral}  zeta functions,
are useful for the clarification of such zeta-based  understanding of the special functions.

Our bilateral zeta function is defined as a periodic function 
which shares certain basic properties of the Barnes zeta function. 
Actually, the bilateral zeta function is defined by a sum of two Barnes multiple zeta functions as follows(see Section 4): 
For $0<\arg(\omega_0)\leq \pi $, 
\begin{equation}
\xi_{r+1}(s, z\mid \omega_0;\,{\boldsymbol{\omega}})=\zeta_{r+1}(s, z+\omega_{0}\mid \omega_{0},\,{\boldsymbol{\omega}} )
+\zeta_{r+1}(s, z\mid -\omega_{0},\,{\boldsymbol{\omega}}). \nonumber
\end{equation}
Here $\zeta_{r+1}(s, z\mid \omega_{0},\,{\boldsymbol{\omega}} )\, ({\boldsymbol{\omega}}:=\,(\omega_{1},\cdots,\omega_{r}) \in \mathbf{C}^r)$ denotes the Barnes multiple zeta function and is given by
\begin{equation}
\zeta_{r+1}(s,\,z\mid \omega_{0},\omega_{1},\cdots,\omega_{r} )
:=\sum_{m_0,m_1,\cdots,m_{r}=0}^{\infty }
 \frac{1}{(z+m_{0}\omega_{0}+m_{1}\omega_{1}+\cdots+m_{r}\omega_{r})^s}.\nonumber
\end{equation} 
For the absolute convergence of these series, certain conditions on the positions of the parameters 
and the variable $z$ should be assumed (see Section 3 and 4 for the precise conditions). 
Under those conditions, the domains of absolute convergence in $s$ for 
$\xi_{r+1}$ and  $\zeta_{r+1}$ are $\Re(s)> r+1$.

The bilateral zeta function inherits many of similar properties of the Barnes zeta function.
Actually, it can be shown that $\xi_{r}$ is continued holomorphically to the whole $s$-plane.   
Moreover, the bilateral zeta function thus defined has obviously 
the Fourier series expansion which is very nice. 
Therefore, we easily find that special functions such as multiple $q$-shifted 
factorials are nicely written by the bilateral zeta function. 
These relationships enable us to derive many important properties of the special functions on the basis of the bilateral zeta functions. 
Moreover, what is remarkable is that one can show the Barnes zeta function itself turns to be expressed by a finite sum of bilateral zeta functions. 
In this way, the Fourier expansion of the Barnes zeta functions and multiple Bernoulli polynomials are transparently derived.

The first main result 
is the following explicit Fourier series expansion(see {Theorem \ref{thm:multiple basic formula}.}):
\begin{equation}
\xi_{r+1}(s,z\mid e^{\pi i};\,{\boldsymbol{\omega}})
=\frac{e^{-\frac{\pi}{2}is}(2{\pi})^s}{\Gamma(s)}
\sum _{n=1}^{\infty }
\frac{n^{s-1}e^{2{\pi}inz}}
{(1-e^{2{\pi}in\omega_{1}}) \cdots(1-e^{2{\pi}in\omega_{r}})},\nonumber
\end{equation}
for all $z,\omega_{1},\cdots,\omega_{r} \in \mathfrak{H}$ and $s \in \mathbf{C}$.
Here $\mathfrak{H}$ is the upper half of the complex plane.

From this Fourier series expansion, small manipulation yields the reflection formula for the multiple gamma function(see corollary \ref{cor:Friedman, Ruijsenaars}) as 
\begin{align}
{} &
\frac{1}{\Gamma_{r+1}(z\mid 1,\,{\boldsymbol{\omega}})\Gamma_{r+1}(1-z\mid 1,\,e^{-{\pi}i}{\boldsymbol{\omega}})}
=\exp\left\{\frac{(-1)^{r+1}{\pi}i}{(r+1)!}B_{r+1,r+1}(z\mid 1,\,{\boldsymbol{\omega}})\right\}(x;\,{\boldsymbol{q}})_{r,\infty }.\nonumber
\end{align}

We notice that the formula above has been obtained by Friedman and Ruijsenaars \cite{FR} 
from Raabe's formula for the integral expression of the Barnes zeta function.
The Fourier series expansion of $\xi_{2}$ gives also simple proofs of 
the inversion formula for the Dedekind $\eta$-function and 
Ramanujan's classical formula concerning the special values of the Riemann zeta function(see Proposition \ref{prop:eta&Ramanujan}).

The second main result is an expression of the Barnes zeta function by a finite sum of bilateral zeta functions in $s(\in \mathbf{C})$ as follows.
\begin{align}
\zeta_{r}(s,z\mid {\boldsymbol{\omega}})
&=\frac{1}{2i\sin({\pi}s)}\nonumber \\
{} & \quad \cdot\! \left\{
\sum _{k=1}^{r}(-1)^{k-1}\xi_{r}(s,\vert {\boldsymbol{\omega}}\vert^{\,+}_{[1,k-1]}+e^{-{\pi}i}z\mid \omega_{k};\widehat{{\boldsymbol{\omega}}}^{-}[k,r](k))\right. \nonumber \\
{} & \quad \left.-\sum _{k=1}^{r}(-1)^{r-k}e^{-{\pi}is}\xi_{r}(s,z+e^{-{\pi}i}\vert {\boldsymbol{\omega}}\vert^{\,+}_{[k+1,r]}\mid \omega_{k};\widehat{{\boldsymbol{\omega}}}^{-}[k,r](k))\right\} \nonumber
\end{align}
for $\omega_{1},\cdots,\omega_{r} \in \mathfrak{H}$ and  
$z\in D:=\{z\in \mathbf{C}^{\ast}\,\vert \,z=\sum _{k=1}^{r}a_{k}\omega_{k}\,\,(0< a_{1},\cdots,a_{r}< 1)\}$.
Here we assume that $\arg(\omega_{j})<\arg(\omega_{k})$($1\leq j<k\leq r$)
(see Section 2 for the definitions of notations).\\

We remark that the proof of this theorem gives a multiple-analogue of the result in Knopp and Robins \cite{KR}.
Further, we obtain also a multiple-analogue of Iseki's formula, a generalization of transformation formula for the theta function \cite{I}. 
This gives also a generalization of infinite product expressions of the multiple sine functions by Narukawa \cite{Na}.
The original proofs of Iseki and that of Narukawa are based on the residues theorem 
while our proof which follows from Kronecker-type limit formula (\ref{eq:f_{+} D1}) for the bilateral zeta function is much simpler.\\
\indent
Throughout the paper, we denote by $\mathbf{N}$ be the set of natural numbers, $\mathbf{Z}$ the ring of rational integers, $\mathbf{Q}$ the field of rational
numbers, $\mathbf{R}$ the field of real numbers, $\mathbf{C}$ the field of complex numbers, 
and put $\mathbf{N}_{0}:=\mathbf{N}\cup \{0\}$,\,$\mathbf{C}^{\ast}:=\mathbf{C}\backslash \{0\}$.

\section{Notations and definitions}
For $c\in \mathbf{C}$, we always assume that
\begin{equation}
-\pi< \arg{c} \leq\pi. \nonumber
\end{equation}
We set $\arg{0}:=0$. 
We note that $i:=\sqrt{-1}\in \mathfrak{H}=\{z\in \mathbf{C}^{\ast}\mid 0<\arg(z)<\pi\}$. \\
\indent
For any vector ${\boldsymbol{X}}=(X_{1},\cdots,X_{r})\in \mathbf{C}^{r}$, we put 
\begin{align}
c{\boldsymbol{X}}:=&\,(cX_{1},\cdots,cX_{r})\in \mathbf{C}^{r}\,(c\in \mathbf{C}), \\
\widehat{{\boldsymbol{X}}}(j):=&\,(X_{1},\cdots,X_{j-1},X_{j+1},\cdots,X_{r})\in \mathbf{C}^{r-1} \\ 
=&\,(X_{1},\cdots,\widehat{X}_{j},\cdots,X_{r}), \\ \nonumber
{\boldsymbol{X}}^{-}[j]:=&\,(X_{1},\cdots,-X_{j},\cdots,X_{r})\in \mathbf{C}^{r}, \\
{\boldsymbol{X}}^{-1}[j]:=&\,(X_{1},\cdots,{X}_{j}^{-1},\cdots,X_{r})\in \mathbf{C}^{r}, \\
{\boldsymbol{X}}^{-1}:=&\,(X_{1}^{-1},\cdots,X_{r}^{-1})\in \mathbf{C}^{r}, \\
\vert {\boldsymbol{X}}\vert^{+}:=&\,X_{1}+\cdots+X_{r}\in \mathbf{C}, \\
\vert{\boldsymbol{X}}\vert^{\times}:=&\,X_{1}\cdots X_{r}\in \mathbf{C}.
\end{align}
When $m,n,j\in \mathbf{N}$ and $1\leq m\leq n\leq r$, we define
{\allowdisplaybreaks
\begin{align}
{\boldsymbol{X}}^{-}[m,n]:=&\,(X_{1},\cdots,{X}_{m-1},-{X}_{m},\cdots,-{X}_{n},{X}_{n+1},\cdots,X_{r})\in \mathbf{C}^{r}, \\
{\boldsymbol{X}}^{-1}[m,n]:=&\,(X_{1},\cdots,X_{m-1},X_{m}^{-1},\cdots,X_{n}^{-1},X_{n+1},\cdots,X_{r})\in \mathbf{C}^{r}, \\
\vert {\boldsymbol{X}}\vert^{+}_{[m,n]}:=&\,\sum _{k=m}^{n}X_{k}\in \mathbf{C}, \\
\vert {\boldsymbol{X}}\vert^{\times}_{[m,n]}:=&\,\prod _{k=m}^{n}X_{k}\in \mathbf{C}. \\
{\boldsymbol{X}}^{-}[r+1,r]:=&\boldsymbol{X}\in \mathbf{C}^{r}, \\
{\boldsymbol{X}}^{-1}[r+1,r]:=&\boldsymbol{X}\in \mathbf{C}^{r}, \\
\vert {\boldsymbol{X}}\vert^{+}_{[1,0]}:=&0, \\
\vert {\boldsymbol{X}}\vert^{+}_{[r+1,r]}:=&0, \\
\vert {\boldsymbol{X}}\vert^{\times}_{[r+1,r]}:=&1.
\end{align}
}

Let $z,\omega_{1},\cdots,\omega_{r}\in \mathbf{C}$ and $\,x:=e^{2\pi{i}z}$. 
Put
\begin{align}
{\boldsymbol{\omega}}:=&\,(\omega_{1},\cdots,\omega_{r})\in \mathbf{C}^{r}, \nonumber \\
{\boldsymbol{q}}:=&\,(q_{1},\cdots,q_{r})\in \mathbf{C}^{r},\nonumber
\end{align}
where $q_{k}:=e^{2\pi{i}\omega_{k}}(k=1,\cdots,r)$.
\begin{dfn}
\label{def:Def1.1}
(1)
For $\omega_{k}\in \mathfrak{H}$ $(k=1,\cdots,r)$, we define a $r$-ple $q$-shifted factorial $\,(x;\,{\boldsymbol{q}})_{r,\infty }$ by
\begin{align}
(x;\,{\boldsymbol{q}})_{r,\infty } &
:=\!\!\!\!\!\!\! \prod 
_{m_1,\cdots,m_{r}=0}^{\infty }
\!\!\!\!\!\!\! {(1-e^{2{\pi}i(m_{1}\omega_{1}+\cdots+m_{r}\omega_{r}+z)})}
=\!\!\!\!\!\!\! \prod 
_{m_1,\cdots,m_{r}=0}^{\infty }
\!\!\!\!\!\!\! {(1-q_{1}^{m_{1}}\cdots q_{r}^{m_{r}}x)}.
\end{align}
If $r=0$, we put 
\begin{equation}
(x)_{0,\infty}:=1-e^{2{\pi}iz}=1-x.
\end{equation}
(2)
For $\Im(\omega_{1}),\cdots,\Im(\omega_{l})<0$ and $\Im(\omega_{l+1}),\cdots,\Im(\omega_{r})>0$,
that is, 
$\vert{q}_{1}\vert,\cdots,\vert{q}_{l}\vert>1$ and $\vert{q}_{l+1}\vert,\cdots,\vert{q}_{r}\vert<1$,
we define the generalized $q$-shifted factorial $\,(\widetilde{x;\,{\boldsymbol{q}}})_{r,\infty }$ by
\begin{align}
(\widetilde{x;\,{\boldsymbol{q}}})_{r,\infty } &
:={(\vert {\boldsymbol{q}}^{-1}\vert^{\times}_{[1,l]}x;\,{\boldsymbol{q}}^{-1}[1,l])^{(-1)^l}_{r,\infty }} \\
{} & \,=\!\!\!\!\!\!\!\!\prod
_{m_1,\cdots,m_{r}=0}^{\infty }
\!\!\!\!\!\!\!\!{(1-q_{1}^{-(m_{1}+1)}\cdots q_{l}^{-(m_{l}+1)}q_{l+1}^{m_{l+1}}\cdots q_{r}^{m_{r}}x)^{(-1)^l}}. \nonumber
\end{align}
\end{dfn}

Put $z_{k}:=z/\omega_{k},\,x_{k}:=e^{2\pi{i}z_{k}},\,\omega_{jk}:=\omega_{j}/\omega_{k}$ and $q_{jk}:=e^{2\pi{i}\omega_{jk}}$.
Define
\begin{align}
{\boldsymbol{\omega}}_{\,k}:=&(\omega_{1k},\cdots,{\omega_{kk}},\cdots,\omega_{rk}),\\
{\boldsymbol{q}}_{\,k}:=&(q_{1k},\cdots,{q_{kk}},\cdots,q_{rk}),\\
\widehat{{\boldsymbol{\omega}}}_{\,k}:=&\widehat{{\boldsymbol{\omega}}_{\,k}}(k)=(\omega_{1k},\cdots,\widehat{\omega}_{kk},\cdots,\omega_{rk}),\\
\widehat{{\boldsymbol{q}}}_{\,k}:=&\widehat{{\boldsymbol{q}}_{\,k}}(k)=(q_{1k},\cdots,\widehat{q}_{kk},\cdots,q_{rk}).
\end{align}
The following lemma is obvious.
\begin{lem}
Assume that $\omega_{1},\cdots,\omega_{r}(\in \mathfrak{H})$
satisfy the order condition {\bf[{ORC}]} which refers as 
\begin{align}
\arg(\omega_{j})<\arg(\omega_{k})\,\,(j<k).\nonumber
\end{align}
Then, for a fixed $k$, one has 
\begin{align}
\label{eq:general q-fact+}
(\widetilde{x_{k};\,\widehat{{\boldsymbol{q}}}}_{\,k})_{r-1,\infty }
{} & \,={(\vert {\boldsymbol{q}}^{\,-1}_{\,k}\vert^{\times}_{[1,k-1]}x_{k};\,\widehat{{\boldsymbol{q}}}^{\,-1}_{\,k}[1,k-1])^{(-1)^{k-1}}_{r-1,\infty }},\\
\label{eq:general q-fact-}
(\widetilde{x_{k}^{-1};\,\widehat{{\boldsymbol{q}}}^{\,-1}_{\,k}})_{r-1,\infty }
{} & \,={(\vert {\boldsymbol{q}}_{\,k}\vert^{\times}_{[k+1,r]}x_{k}^{-1};\,\widehat{{\boldsymbol{q}}}^{\,-1}_{\,k}[1,k-1])^{(-1)^{r-k}}_{r-1,\infty }}.
\end{align}
\end{lem}

\section{Barnes zeta functions}
In this section, we recall various properties of the Barnes zeta function from Barnes \cite{Ba}. 
Let $r\in \mathbf{N}_{0}$. 
For $\Re(s)>r+1$, we define ($r+1$)-ple Barnes zeta functions $\zeta_{r+1}(s,\,z\mid \omega_0,\,{\boldsymbol{\omega}} )$ by the series
\begin{equation}
\zeta_{r+1}(s,\,z\mid \omega_0,\,{\boldsymbol{\omega}} )
:=\sum_{m_0,\cdots,m_{r}=0}^{\infty }
 \frac{1}{(z+m_{0}\omega_{0}+m_{1}\omega_{1}+\cdots+m_{r}\omega_{r})^s}.
\end{equation} 
Here $z,\omega_0$ and $\boldsymbol{\omega}$ satisfy the following one-side condition {\bf[{OC}]}. 
\begin{equation}
\max\{\arg(z),\arg(\omega_0),\,\arg(\omega_{1}),\cdots,\arg(\omega_{r})\}
-\min\{\arg(z),\arg(\omega_0),\,\arg(\omega_{1}),\cdots,\arg(\omega_{r})\}<\pi. \nonumber
\end{equation}

Throughout this paper, when we consider the Barnes zeta function,
we always assume that $z,\omega_0$ and ${\boldsymbol{\omega}}$ satisfy the condition {\bf[{OC}]}.
For the convenience, we set
\begin{equation}
\zeta_{0}(s,\,z):=z^{-s}.
\end{equation}
The Barnes zeta function $\zeta_{r+1}(s,\,z\mid \omega_0,\,{\boldsymbol{\omega}} )$ converges absolutely and uniformly for any compact set in the domain ${\Re(s)>r+1}$.
It is well known that $\zeta_{r+1}(s,\,z\mid \omega_0,\,{\boldsymbol{\omega}} )$ is 
continued meromorphically to the whole plane $\mathbf{C}$.

\begin{lem}
\label{prop:multiplication}
If $\alpha\in \mathbf{C}^{\ast}$ satisfies the conditions 
\begin{align}
-\pi< \arg(\alpha)+\arg(z)\leq \pi, 
\,\,\,\,-\pi< \arg(\alpha)+\arg(\omega_{j})\leq \pi \,\,(1\leq j\leq r), 
\end{align}
then the following equality holds.
\begin{equation}
\label{eq:Barnes multiplication}
\zeta_{r}(s,\alpha{z}\mid \alpha{\boldsymbol{\omega}})
=\alpha^{-s}\zeta_{r}(s,{z}\mid {\boldsymbol{\omega}}).
\end{equation}
\end{lem}
\begin{proof}
It suffices to show the equation (\ref{eq:Barnes multiplication}) when $\Re(s)>r$.
For $\alpha\in \mathbf{C}^{\ast}$ satisfying the condition, we have
\begin{align}
\zeta_{r}(s,\alpha{z}\mid \alpha{\boldsymbol{\omega}})
=&\sum_{m_1,\cdots,m_{r}=0}^{\infty }
 \frac{1}{(\alpha{z}+m_{1}\alpha{\omega_{1}}+\cdots+m_{r}\alpha{\omega_{r}})^s} \nonumber \\
=&\sum_{m_1,\cdots,m_{r}=0}^{\infty }
 \frac{1}{\alpha^{s}({z}+m_{1}\omega_{1}+\cdots+m_{r}\omega_{r})^s}
=\alpha^{-s}\zeta_{r}(s,z\mid {\boldsymbol{\omega}}).\nonumber
\end{align}
\end{proof}

\begin{lem}
\label{prop:Barnes zeta}
{\rm(1)}
We have 
\begin{equation}
\label{eq:Barnes zeta relation}
\zeta_{r}(s, z+\omega_k\mid {\boldsymbol{\omega}})
=\zeta_{r}(s, z\mid {\boldsymbol{\omega}})-\zeta_{r-1}(s, z\mid \widehat{{\boldsymbol{\omega}}}(k)),\,\,\,\,(k=1,\cdots,r).
\end{equation}
{\rm(2)}
Let $X_{r}:=\{-m_{1}\omega_{1}-\cdots-m_{r}\omega_{r}\mid m_{1},\cdots,m_{r}\in \mathbf{N}_{0}\}$ and put $X_{0}:=\{0\}$.
The function $\zeta_{r}(s, z\mid {\boldsymbol{\omega}})$ is continued analytically to $\mathbf{C}\backslash X_{r}$ as a multivalued holomorphic function in $z$.
\end{lem}
\begin{proof}
(1) :
The relation follows immediately from the definition.\\
(2) :
By Lemma \ref{prop:multiplication}, it suffices to show the assertion when $\Re(\omega_{k})>0(k=1,\cdots,r)$.
By the relation $\zeta_{1}(s, z+1\mid 1)=\zeta_{1}(s, z\mid 1)-z^{-s}$, we have for all $n\in \mathbf{N}$
\begin{equation}
\zeta_{1}(s, z\mid 1)=\zeta_{1}(s, z+n\mid 1)+\sum_{m=0}^{n-1}(z+m)^{-s}.\nonumber
\end{equation}
Hence, one finds that $\zeta_{1}(s, z\mid 1)$ is continued analytically to the region $\{z\in \mathbf{C}\mid \Re(z)>-n\}\backslash X_{1}$ as a holomorphic function in $z$.
This proves the case $r=1$.\\
\indent
For $r\geq 2$, by (\ref{eq:Barnes zeta relation}), we have
\begin{equation}
\zeta_{r}(s, z\mid {\boldsymbol{\omega}})
=\zeta_{r}(s, z+n\omega_{r}\mid {\boldsymbol{\omega}})
+\sum_{m=0}^{n-1}\zeta_{r-1}(s, z+m\omega_{r}\mid \widehat{{\boldsymbol{\omega}}}(r)).\nonumber
\end{equation}
By induction, we observe that $\zeta_{r}(s, z\mid {\boldsymbol{\omega}})$ is continued analytically to $\{z\in \mathbf{C}\mid \Re(z)>-n\Re(\omega_{r})\}\backslash X_{r}$ for all $n\in \mathbf{N}$.
Hence the desired claim follows.
\end{proof}
In order to describe special values of the Barnes zeta function, we recall the multiple Bernoulli polynomials \cite{Ba}. 
We follow the notational convention in \cite{Na}.
\begin{dfn}
We define the multiple Bernoulli polynomials $B_{r,n}(z\mid {\boldsymbol{\omega}})$ by a generating function as
\begin{equation}
\frac{t^{r}e^{zt}}{(e^{\omega_{1}t}-1)\cdots(e^{\omega_{r}t}-1)}
=\sum_{k=0}^{\infty}B_{r,k}(z\mid {\boldsymbol{\omega}})\frac{t^k}{k!}. 
\end{equation}
Here $z,\omega_{1},\cdots,\omega_{r}$ do not necessary satisfy the condition {\bf[{OC}]}.
\end{dfn}
The multiple Bernoulli polynomial $B_{r,n}(z\mid {\boldsymbol{\omega}})$ is actually a polynomial of degree $n$ in $z$ and is symmetric in $\omega_{1},\cdots,\omega_{r}$.
\begin{exa}
\begin{align}
\label{eq:B_{r,0}}
B_{r,0}(z\mid {\boldsymbol{\omega}})=&\frac{1}{\vert {\boldsymbol{\omega}}\vert},\\
\label{eq:B_{2,2}}
B_{2,2}(z\mid\omega_{1},\omega_{2})=&
\frac{1}{\omega_{1}\omega_{2}}z^2-\frac{\omega_{1}+\omega_{2}}{\omega_{1}\omega_{2}}z+\frac{\omega_{1}^2+\omega_{2}^2+3\omega_{1}\omega_{2}}{6\omega_{1}\omega_{2}}.
\end{align}
\end{exa}
The following lemma is essentially due to Barnes and can be seen in
 \cite{Na} (see the formulas (12)-(17)).
\begin{lem}
It holds that 
\label{prop:multiple Brnoulli1}
\begin{align}
\label{eq:multiple Brnoulli1}
B_{r,n}(cz\mid c{\boldsymbol{\omega}})=&\,c^{n-r}B_{r,n}(z\mid {\boldsymbol{\omega}}) \,\,\, (c\in \mathbf{C}^{\ast}), \\
\label{eq:multiple Brnoulli2}
B_{r,n}(\vert {\boldsymbol{\omega}}\vert^{+}-z\mid {\boldsymbol{\omega}})=&\,(-1)^{n}B_{r,n}(z\mid {\boldsymbol{\omega}}), \\
\label{eq:multiple Brnoulli3}
B_{r,n}(z+\omega_{j}\mid {\boldsymbol{\omega}})-B_{r,n}(z\mid {\boldsymbol{\omega}})=&\,nB_{r-1,n-1}(z\mid \widehat{{\boldsymbol{\omega}}}(j)), \\
\label{eq:multiple Brnoulli4}
B_{r,n}(z\mid {\boldsymbol{\omega}}^{-}[j])=&\,-B_{r,n}(z+\omega_{j}\mid {\boldsymbol{\omega}}), \\
\label{eq:multiple Brnoulli5}
B_{r,n}(z\mid {\boldsymbol{\omega}})+B_{r,n}(z\mid {\boldsymbol{\omega}}^{-}[j])=&\,-nB_{r-1,n-1}(z\mid \widehat{{\boldsymbol{\omega}}}(j)), \\
\label{eq:multiple Brnoulli6}
\frac{d}{dz}B_{r,n}(z\mid {\boldsymbol{\omega}})=&\,nB_{r,n-1}(z\mid {\boldsymbol{\omega}}).
\end{align}
\end{lem}
The special values of the Barnes multiple zeta functions are given by the multiple Bernoulli polynomial as follows.
\begin{lem}
{\rm(\cite{Ba})}
{\rm(1)}
For all $m\in \mathbf{N}$,
\begin{equation}
\label{eq:BarnesSV}
\zeta_{r}(1-m,\,z\mid {\boldsymbol{\omega}})=(-1)^{r}\frac{(m-1)!}{(m+r-1)!}B_{r,r+m-1}(z\mid {\boldsymbol{\omega}}).
\end{equation}
{\rm(2)}
For $m=1,\cdots,r$,
\begin{equation}
\label{eq:BarnesSV2}
\mathop{\rm Res}_{s=m}\zeta_{r}(s,\,z\mid {\boldsymbol{\omega}})ds
=\frac{(-1)^{r-m}}{(m-1)!(r-m)!}B_{r,r-m}(z\mid {\boldsymbol{\omega}}).
\end{equation}
\end{lem}

\section{Bilateral zeta functions}
\begin{dfn}
\label{dfn:bilateral dfn}
Let $r \in \mathbf{N}_{0}$. Assume that $z,\omega_0$ and ${\boldsymbol{\omega}}=(\omega_{1},\cdots,\omega_{r})$ satisfy the strong one-side condition {\bf[{SOC}]}, that is given by 
\begin{align}
\min\{\arg(\pm\omega_0)\} \leq& \arg(z)\leq \max\{\arg(\pm\omega_0)\}, \nonumber \\
\min\{\arg(\pm\omega_0)\}<&\arg(\omega_{j})<\max\{\arg(\pm\omega_0)\}\,\,(1\leq j\leq r). \nonumber
\end{align}

%
%

For $\Re(s)>r+1$, we define the bilateral ($r+1$)-ple zeta function $\xi_{r+1}$ as follows.
\begin{align}
\label{eq:def bilateral}
\xi_{r+1}(s, z\mid \omega_0{'};\,{\boldsymbol{\omega}})
:=&\zeta_{r+1}(s, z+\omega_{0}\mid \omega_{0},\,{\boldsymbol{\omega}} )
+\zeta_{r+1}(s, z\mid -\omega_{0},\,{\boldsymbol{\omega}}).
\end{align}
Here, we put $\omega_0{'}:=\vert \omega_0 \vert e^{i\max\{\arg(\pm\omega_0)\}}$.
\end{dfn}
Throughout this paper, we assume that $0< \arg(\omega_0)\leq \pi$. Hence we may write
\begin{align}
\xi_{r+1}(s, z\mid \omega_0{'};\,{\boldsymbol{\omega}})
=&\xi_{r+1}(s, z\mid \omega_0;\,{\boldsymbol{\omega}}) \nonumber \\
=&\zeta_{r+1}(s, z+\omega_{0}\mid \omega_{0},\,{\boldsymbol{\omega}} )
+\zeta_{r+1}(s, z\mid e^{-\pi i}\omega_{0},\,{\boldsymbol{\omega}}). \nonumber
\end{align}
In addition, when we consider the bilateral zeta function, we always assume that 
$z,\omega_0,\,{\boldsymbol{\omega}}$ satisfy the condition {\bf[{SOC}]}.
\begin{lem}
\label{prop:bilateral zeta1}
{\rm(1)}
The series expression of the bilateral zeta function $\xi_{r+1}(s, z\mid \omega_0;\,{\boldsymbol{\omega}})$ converges absolutely in the domain ${\Re(s)>r+1}$.\\
{\rm(2)}
It holds that
\begin{align}
\xi_{r+1}(s, z\mid \omega_{0};\,{\boldsymbol{\omega}}) & 
=\zeta_{r+1}(s, z\mid \omega_{0},\,{\boldsymbol{\omega}} )
+\zeta_{r+1}(s, z+e^{-\pi i}\omega_{0}\mid e^{-\pi i}\omega_{0},\,{\boldsymbol{\omega}} )\\
{} & 
=\zeta_{r+1}(s, z\mid \omega_{0},\,{\boldsymbol{\omega}} )
+\zeta_{r+1}(s, z\mid e^{-\pi i}\omega_{0},\,{\boldsymbol{\omega}} )
-\zeta_{r}(s, z\mid {\boldsymbol{\omega}} ).
\end{align}
{\rm(3)} The bilateral zeta function $\xi_{r+1}(s, z\mid \omega_{0};\,{\boldsymbol{\omega}})$ is continued analytically to the whole $s$-plane $\mathbf{C}$. 
Moreover $\xi_{r+1}(s, z\mid \omega_{0};\,{\boldsymbol{\omega}})$ is continued analytically to a multivalued holomorphic function of $z$ in 
$\mathbf{C}\backslash (X_{r}\cup\{n\omega_{0}\mid n\in \mathbf{Z}\})$.
\end{lem}
\begin{proof}
(1), (2) :
 These are obvious from the definition and (\ref{eq:Barnes zeta relation}) of Lemma \ref{prop:Barnes zeta}.\\
(3) : The first assertion follows from the relations above and analytic continuation of the Barnes zeta function.
By Lemma \ref{prop:Barnes zeta}, we see that $\zeta_{r+1}(s, z\mid \omega_{0},\,{\boldsymbol{\omega}} ), 
\zeta_{r+1}(s, z\mid e^{-\pi i}\omega_{0},\,{\boldsymbol{\omega}} )$ and $\zeta_{r}(s, z\mid {\boldsymbol{\omega}} )$ 
are continued analytically to multivalued holomorphic functions of $z$ in $\mathbf{C}\backslash (X_{r}\cup\{-n\omega_{0}\mid n\in \mathbf{N}_{0}\})$,
\,$\mathbf{C}\backslash (X_{r}\cup\{n\omega_{0}\mid n\in \mathbf{N}_{0}\})$ and $\mathbf{C}\backslash X_{r}$ respectively.
\end{proof}
\begin{lem}
{\rm(1)}
It holds that 
\begin{equation}
\xi_{r+1}(s, z+\omega_0\mid \omega_0;\,{\boldsymbol{\omega}})=\xi_{r+1}(s, z\mid \omega_0;\,{\boldsymbol{\omega}}).
\end{equation}
{\rm(2)}
For $k=1,\cdots,r$,
\begin{equation}
\xi_{r+1}(s, z+\omega_k\mid \omega_0;\,{\boldsymbol{\omega}})=\xi_{r+1}(s, z\mid \omega_0;\,{\boldsymbol{\omega}})
-\xi_{r}(s, z\mid \omega_0;\,\widehat{{\boldsymbol{\omega}}}(k)).
\end{equation}
\end{lem}
\begin{proof}
(1) 
By (\ref{eq:def bilateral}) and Lemma \ref{prop:bilateral zeta1}, 
\begin{align}
\xi_{r+1}(s, z+\omega_0\mid \omega_0;\,{\boldsymbol{\omega}})
 & 
=\zeta_{r+1}(s, z+\omega_{0}\mid \omega_{0},\,{\boldsymbol{\omega}} )
+\zeta_{r+1}(s, (z+\omega_{0})+e^{-\pi i}\omega_{0}\mid e^{-\pi i}\omega_{0},\,{\boldsymbol{\omega}} ) \nonumber \\
{} & 
=\zeta_{r+1}(s, z+\omega_{0}\mid \omega_{0},\,{\boldsymbol{\omega}} )
+\zeta_{r+1}(s, z\mid e^{-\pi i}\omega_{0},\,{\boldsymbol{\omega}} )\nonumber \\
{} & 
=\xi_{r+1}(s, z\mid \omega_0;\,{\boldsymbol{\omega}}).\nonumber
\end{align}
(2) 
It is obvious from the definition and Lemma \ref{prop:Barnes zeta}.
\end{proof}
\begin{lem}
\label{prop:bilateral multiplication}
If $\alpha\in \mathbf{C}^{\ast}$ satisfies the conditions
\begin{align}
\label{eq:bilateral multiplication condition}
-\pi<& \arg(\alpha)+\arg(z)\leq \pi, \nonumber \\
0<& \arg(\alpha)+\arg(\omega_{0})\leq \pi,  \\
-\pi<& \arg(\alpha)+\arg(\omega_{j})\leq \pi\,\,(1\leq j\leq r), \nonumber
\end{align}
then the following equality holds.
\begin{equation}
\label{eq:bilateral multiplication}
\xi_{r+1}(s,\alpha{z}\mid \alpha{\omega_{0}};\alpha{{\boldsymbol{\omega}}})
=\alpha^{-s}\xi_{r+1}(s,z\mid \omega_{0};\,{\boldsymbol{\omega}}).
\end{equation}
\end{lem}
\begin{proof}
By (\ref{eq:Barnes multiplication}) of Lemma \ref{prop:multiplication}, we obtain 
\begin{align}
\xi_{r+1}(s,\alpha{z}\mid \alpha{\omega_{0}};\alpha{{\boldsymbol{\omega}}})
&=\zeta_{r+1}(s, \alpha(z+\omega_{0})\mid \alpha{\omega_{0}},\,\alpha{{\boldsymbol{\omega}}})
+\zeta_{r+1}(s, \alpha{z}\mid \alpha{e^{-\pi i}\omega_{0}},\,\alpha{{\boldsymbol{\omega}}})\nonumber \\
&=\alpha^{-s}(\zeta_{r+1}(s, z+\omega_{0}\mid \omega_{0},\,{\boldsymbol{\omega}})
+\zeta_{r+1}(s, z\mid e^{-\pi i}\omega_{0},\,{\boldsymbol{\omega}}))\nonumber \\
&=\alpha^{-s}\xi_{r+1}(s,z\mid \omega_{0};\,{\boldsymbol{\omega}}).\nonumber
\end{align}
\end{proof}
\begin{cor}
\label{cor:bilateral multiplication2}
One has
\begin{equation}
\xi_{r+1}(s,z\mid \omega_{0};\,{\boldsymbol{\omega}})
=\left(\frac{e^{{\pi}i}}{\omega_{0}}\right)^s
\xi_{r+1}(s,e^{{\pi}i}z_{0}\mid e^{{\pi}i};\,e^{{\pi}i}{\boldsymbol{\omega}}_{\,0}).
\end{equation}
\end{cor}
\begin{proof}
By the condition [SOC], $\alpha=\frac{e^{{\pi}i}}{\omega_{0}}$ satisfies 
(\ref{eq:bilateral multiplication condition}) in Lemma \ref{prop:bilateral multiplication}. 
Therefore, we have 
\begin{equation}
\xi_{r+1}(s,e^{{\pi}i}z_{0}\mid e^{{\pi}i};\,e^{{\pi}i}{\boldsymbol{\omega}}_{\,0})
=\left(\frac{e^{{\pi}i}}{\omega_{0}}\right)^{-s}\xi_{r+1}(s,z\mid \omega_{0};\,{\boldsymbol{\omega}}). \nonumber
\end{equation}
\end{proof}
\begin{lem}
\label{prop:basic formula}
If $z\in \mathfrak{H}$, for all $s\in\mathbf{C}$, one has 
\begin{equation}
\label{eq:basic formula}
\xi_{1}(s,z\mid e^{\pi i})
=\frac{e^{-\frac{\pi}{2}is}(2\pi)^s}{\Gamma(s)}
\sum _{n=1}^{\infty }
n^{s-1}e^{2{\pi}inz}.
\end{equation} 
\end{lem}
\begin{proof}
We may assume $\Re(s)>1$.
Since $z\in \mathfrak{H}$, we observe
\begin{align}
\xi_{1}(s,z\mid e^{\pi i})
=&\zeta_{1}(s, z+e^{\pi i}\mid e^{\pi i})
+\zeta_{1}(s, z\mid 1) \nonumber \\
=&\sum_{n=0}^{\infty }\frac{1}{(n{e^{\pi i}}+e^{\pi i}+z)^s}
+\sum_{n=0}^{\infty }\frac{1}{(n+z)^s}
=\frac{e^{-\frac{\pi}{2}is}(2\pi)^s}{\Gamma(s)}
\sum _{n=1}^{\infty }
n^{s-1}e^{2{\pi}inz},\nonumber
\end{align}
where we have used the Lipschitz's formula (\cite{AAR}) for the third equality.
\end{proof}
\begin{thm}
\label{thm:multiple basic formula}
If $z,\omega_{1},\cdots,\omega_{r}\in \mathfrak{H}$, for all $s\in\mathbf{C}$, one has
\begin{equation}
\label{eq:multiple basic formula}
\xi_{r+1}(s,z\mid e^{\pi i};\,{\boldsymbol{\omega}})
=\frac{e^{-\frac{\pi}{2}is}(2\pi)^s}{\Gamma(s)}
\sum _{n=1}^{\infty }
\frac{n^{s-1}e^{2{\pi}inz}}
{(1-e^{2{\pi}in\omega_{1}}) \cdots(1-e^{2{\pi}in\omega_{r}})}.
\end{equation}
\end{thm}
\begin{proof}
It is enough to show the assertion when $\Re(s)>r+1$. 
Since the series for $\xi_{r+1}(s,z\mid e^{\pi i};\,{\boldsymbol{\omega}})$ converges absolutely, we observe 
\begin{align}
\xi_{r+1}(s,z\mid e^{\pi i};\,{\boldsymbol{\omega}}) =&
\sum_{m_1,\cdots,m_{r}=0}^{\infty }
\xi_{1}(s,z+m_{1}\omega_{1}+\cdots+m_{r}\omega_{r}\mid e^{\pi i}) \nonumber \\
=&\sum_{m_1,\cdots,m_{r}=0}^{\infty }
\frac{e^{-\frac{\pi}{2}is}(2\pi)^s}{\Gamma(s)}
\sum _{n=1}^{\infty }
n^{s-1}e^{2{\pi}in(z+m_{1}\omega_{1}+\cdots+m_{r}\omega_{r})} \nonumber \\
=&\frac{e^{-\frac{\pi}{2}is}(2\pi)^s}{\Gamma(s)}
\sum _{n=1}^{\infty }
\frac{n^{s-1}e^{2{\pi}inz}}
{(1-e^{2{\pi}in\omega_{1}}) \cdots(1-e^{2{\pi}in\omega_{r}})}. \nonumber
\end{align}
Since $z,\omega_{1},\cdots,\omega_{r}\in \mathfrak{H}$, the second equality follows from (\ref{eq:basic formula}) immediately.
\end{proof}
It is easy to show the following corollary from Lemma \ref{cor:bilateral multiplication2} and Theorem \ref{thm:multiple basic formula}.
\begin{cor}
\label{cor:bilateral zero}
{\rm(1)}
The bilateral zeta function $\xi_{r+1}(s,z\mid \omega_{0};\,{\boldsymbol{\omega}})$ is an entire function in $s \in \mathbf{C}$.\\
{\rm(2)}
For all $m \in \mathbf{N}$, one has
\begin{equation}
\label{eq:bilateral zero}
\xi_{r+1}(1-m,z\mid \omega_{0};\,{\boldsymbol{\omega}})=0.
\end{equation}
\end{cor}
\begin{cor}
\label{cor:q-fact}
If $z,\omega_{1},\cdots,\omega_{r}\in \mathfrak{H}$, for all $m\in \mathbf{N}$, one has
\begin{equation}
\frac{\partial \xi_{r+1}}{\partial s}
(1-m,z\mid e^{\pi i};\,{\boldsymbol{\omega}})
=\frac{(m-1)!}{(2{\pi}i)^{m-1}}
\sum _{n=1}^{\infty }
\frac{e^{2{\pi}inz}}
{n^m(1-e^{2{\pi}in\omega_{1}}) \cdots(1-e^{2{\pi}in\omega_{r}})}.
\end{equation}
In particular, 
\begin{equation}
\label{eq:bilateral q-fact}
\exp\left(-\frac{\partial \xi_{r+1}}{\partial s}
(0,z\mid e^{\pi i};\,{\boldsymbol{\omega}})\right)
=(x;\,{\boldsymbol{q}})_{r,\infty }.
\end{equation} 
\end{cor}
\begin{proof}
If $z,\omega_{1},\cdots,\omega_{r}\in \mathfrak{H}$, then
\begin{align}
\frac{\partial \xi_{r+1}}{\partial s}
(1-m,z\mid e^{\pi i};\,{\boldsymbol{\omega}})
&=\lim_{s\to 0}\frac{\xi_{r+1}(s+1-m,z\mid e^{\pi i};\,{\boldsymbol{\omega}})-\xi_{r+1}(1-m,z\mid e^{\pi i};\,{\boldsymbol{\omega}})}{s}\nonumber \\
&=\lim_{s\to 0}\frac{\xi_{r+1}(s+1-m,z\mid e^{\pi i};\,{\boldsymbol{\omega}})}{s} \nonumber \\
&=\frac{(m-1)!}{(2{\pi}i)^{m-1}}
\sum _{n=1}^{\infty }
\frac{e^{2{\pi}inz}}
{n^{m}(1-e^{2{\pi}in\omega_{1}}) \cdots(1-e^{2{\pi}in\omega_{r}})}.\nonumber
\end{align}
The second equality follows from (\ref{eq:bilateral zero}) and the third one from (\ref{eq:multiple basic formula}).
Moreover, by analytic continuation, it is enough to show the assertion (\ref{eq:bilateral q-fact}) when $z \in \mathfrak{H}$.
\begin{align}
\exp\left(-\frac{\partial \xi_{r+1}}{\partial s}
(0,z\mid e^{\pi i};\,{\boldsymbol{\omega}})\right)
&=\exp\left(-\sum _{n=1}^{\infty }
\frac{e^{2{\pi}inz}}
{n(1-e^{2{\pi}in\omega_{1}}) \cdots(1-e^{2{\pi}in\omega_{r}})}\right)\nonumber \\
&=\exp\left(-\sum _{n=1}^{\infty }\sum_{m_1,\cdots,m_{r}=0}^{\infty }
\frac{e^{2{\pi}in(z+m_{1}\omega_{1}+\cdots+m_{r}\omega_{r})}}{n}\right) \nonumber \\
&=\exp\left(\log \prod _{m_1,\cdots,m_{r}=0}^{\infty }
{(1-e^{2{\pi}i(m_{1}\omega_{1}+\cdots+m_{r}\omega_{r}+z)})}\right)\nonumber \\
&=(x;\,{\boldsymbol{q}})_{r,\infty }.\nonumber
\end{align}
This proves (\ref{eq:bilateral q-fact}).
\end{proof}
\begin{cor}
\label{cor:Friedman, Ruijsenaars}
{\rm(\cite{FR})}
If $\omega_{1},\cdots,\omega_{r}\in \mathfrak{H}$, then
\begin{equation}
\label{eq:Friedman-Ruijsenaars}
\frac{1}{\Gamma_{r+1}(z\mid 1,\,{\boldsymbol{\omega}})\Gamma_{r+1}(1-z\mid 1,\,e^{-{\pi}i}{\boldsymbol{\omega}})}
=\exp\left\{\frac{(-1)^{r+1}{\pi}i}{(r+1)!}B_{r+1,r+1}(z\mid 1,\,{\boldsymbol{\omega}})\right\}(x;\,{\boldsymbol{q}})_{r,\infty }.
\end{equation}
Here $\Gamma_{r}(z\mid {\boldsymbol{\omega}})$ is the multiple gamma function defined by
\begin{equation}
{\Gamma_{r}(z\mid {\boldsymbol{\omega}})}:=\exp \left(\frac{\partial \zeta_{r}}{\partial s}(0,z\mid 
{\boldsymbol{\omega}})\right).
\end{equation}
\end{cor}
\begin{proof}
Suppose $\omega_{1},\cdots,\omega_{r}\in \mathfrak{H}$.
Since $\Gamma_{r}(z\mid {\boldsymbol{\omega}})^{-1}$ and $\Gamma_{r+1}(1-z\mid 1,\,e^{-{\pi}i}{\boldsymbol{\omega}})^{-1}$ are continued holomorphically to the whole $z$-plane, it is enough to show the assertion when $z \in \mathfrak{H}$.
By Corollary \ref{cor:q-fact}, if $z,\omega_{1},\cdots,\omega_{r}\in \mathfrak{H}$,
\begin{equation}
(x;\,{\boldsymbol{q}})_{r,\infty }
=\exp\left(-\frac{\partial \xi_{r+1}}{\partial s}
(0,z\mid e^{\pi i};\,{\boldsymbol{\omega}})\right).\nonumber \\
\end{equation}
By (\ref{eq:def bilateral}), we observe
\begin{align}
\xi_{r+1}(s,z\mid e^{{\pi}i};\,{\boldsymbol{\omega}})
&=\zeta_{r+1}(s,z\mid 1,\,{\boldsymbol{\omega}})+\zeta_{r+1}(s,z+e^{{\pi}i}\mid e^{{\pi}i},\,{\boldsymbol{\omega}})\nonumber \\
&=\zeta_{r+1}(s,z\mid 1,\,{\boldsymbol{\omega}})+e^{-{\pi}is}\zeta_{r+1}(s,e^{-{\pi}i}z+1\mid 1,\,e^{-{\pi}i}{\boldsymbol{\omega}}).\nonumber
\end{align}
Therefore, we have 
\begin{align}
(x;\,{\boldsymbol{q}})_{r,\infty }
&=\exp({\pi}i\zeta_{r+1}(0,e^{-{\pi}i}z+1\mid 1,\,e^{-{\pi}i}{\boldsymbol{\omega}}))\nonumber \\
{} & \quad \cdot\exp\left(-\frac{\partial \zeta_{r+1}}{\partial s}(0,z\mid 1,\,{\boldsymbol{\omega}})
-\frac{\partial \zeta_{r+1}}{\partial s}(0,e^{-{\pi}i}z+1\mid 1,\,e^{-{\pi}i}{\boldsymbol{\omega}})\right)\nonumber \\
&=\exp\left(\frac{-(-1)^{r+1}{\pi}i}{(r+1)!}B_{r+1,r+1}(z\mid 1,\,{\boldsymbol{\omega}})\right)
\frac{1}{\Gamma_{r+1}(z\mid 1,\,{\boldsymbol{\omega}})\Gamma_{r+1}(1-z\mid 1,\,e^{-{\pi}i}{\boldsymbol{\omega}})}.\nonumber
\end{align}
The last equality follows from (\ref{eq:BarnesSV}) and Lemma \ref{prop:multiple Brnoulli1} (\ref{eq:multiple Brnoulli2}), (\ref{eq:multiple Brnoulli4}).
\end{proof}

Now we give the main theorem of this paper.
\begin{thm}
\label{thm:bilateral relation}
Suppose that $r\geq 2$. Assume that $\omega_{1},\cdots,\omega_{r}\in \mathfrak{H}$ satisfy the condition {\bf[{ORC}]}.
Put
{\allowdisplaybreaks
\begin{align}
D &
:=\left\{z\in \mathbf{C}^{\ast}\,\Bigg\vert \,z=\sum _{k=1}^{r}a_{k}\omega_{k}\,\,(0< a_{1},\cdots,a_{r}< 1)\right\}, \nonumber \\
D_{+} & 
:=\{z\in \mathbf{C}^{\ast}\mid \arg(\omega_{r})< \arg(z)<\pi\}, \nonumber \\
D_{-} &
:=\{z\in \mathbf{C}^{\ast}\mid 0<\arg(z)< \arg(\omega_{1})\}, \nonumber \\
f_{+}(s,z\mid {\boldsymbol{\omega}}) &
:=\zeta_{r}(s,e^{-{\pi}i}z\mid e^{-{\pi}i}{\boldsymbol{\omega}})
+(-1)^{r-1}\zeta_{r}(s,\vert {\boldsymbol{\omega}}\vert^{+}+e^{-{\pi}i}z\mid {\boldsymbol{\omega}}),\nonumber \\
f_{-}(s,z\mid {\boldsymbol{\omega}}) &
:=\zeta_{r}(s,z\mid {\boldsymbol{\omega}})
+(-1)^{r-1}\zeta_{r}(s, z+e^{-{\pi}i}\vert {\boldsymbol{\omega}}\vert^{+} \mid e^{-{\pi}i}{\boldsymbol{\omega}}).\nonumber
\end{align}
}
{\rm(1)}
If $z\in D\cup D_{+}$, one has for all $s\in \mathbf{C}$,
{\allowdisplaybreaks
\begin{align}
f_{+}(s,z\mid {\boldsymbol{\omega}})
\label{eq:f_{+}xi}
&=\sum _{k=1}^{r}(-1)^{k-1}\xi_{r}(s,\vert {\boldsymbol{\omega}}\vert^{\,+}_{[1,k-1]}+e^{-{\pi}i}z\mid \omega_{k};\widehat{{\boldsymbol{\omega}}}^{-}[k,r](k)) \\
\label{eq:f_{+}basic}
&=\frac{e^{\frac{\pi}{2}is}(2\pi)^s}{\Gamma(s)}\sum _{k=1}^{r}\omega_{k}^{-s}\sum _{n=1}^{\infty }
n^{s-1}e^{2{\pi}inz_{k}}\!\!\prod _{j=1,\,j\neq k}^{r}\!\!(1-e^{2{\pi}in\omega_{jk}})^{-1}.
\end{align}
}
{\rm(2)}
If $z\in D\cup D_{-}$, one has for all $s\in \mathbf{C}$,
\begin{align}
f_{-}(s,z\mid \boldsymbol{\omega})
\label{eq:f_{-}xi}
&=\sum _{k=1}^{r}(-1)^{r-k}\xi_{r}(s,z+e^{-{\pi}i}\vert {\boldsymbol{\omega}}\vert^{\,+}_{[k+1,r]}\mid \omega_{k};\widehat{{\boldsymbol{\omega}}}^{-}[k,r](k)) \\
\label{eq:f_{-}basic}
&=\frac{e^{\frac{\pi}{2}is}(2\pi)^s}{\Gamma(s)}\sum _{k=1}^{r}\omega_{k}^{-s}\sum _{n=1}^{\infty }
n^{s-1}e^{-2{\pi}inz_{k}}\!\!\prod _{j=1,\,j\neq k}^{r}\!\!(1-e^{-2{\pi}in\omega_{jk}})^{-1},
\end{align}
\end{thm}
To prove the main theorem, we need the following lemma.
\begin{lem}
\label{lem:f_{+}f_{-}uhp}
Suppose that $\omega_{1},\cdots,\omega_{r}$ satisfy the condition {\bf[{ORC}]}.\\
{\rm(1)}
If $z\in D\cup D_{+}$, then
\begin{equation}
\label{eq:f_{+}uhp}
z_{k}+e^{{\pi}i}\vert {\boldsymbol{\omega}}_{\,k}\vert^{\,+}_{[1,k-1]}\in \mathfrak{H}\,\,\,\,(k=1,\cdots,r).
\end{equation}
{\rm(2)}
If $z\in D\cup D_{-}$, then 
\begin{equation}
\label{eq:f_{-}uhp}
e^{{\pi}i}z_{k}+\vert {\boldsymbol{\omega}}_{\,k}\vert^{\,+}_{[k+1,r]}\in \mathfrak{H}\,\,\,\,(k=1,\cdots,r).
\end{equation}
\end{lem}
\begin{proof}
(1) : 
The condition {\bf[{ORC}]} shows
\begin{equation}
e^{{\pi}i}\omega_{1\,k},\cdots,e^{{\pi}i}\omega_{k-1\,k},\omega_{k+1\,k},\cdots,\omega_{r\,k}\in \mathfrak{H}.\nonumber
\end{equation}
Therefore we have $e^{{\pi}i}\vert \boldsymbol{\omega}_{\,k}\vert^{\,+}_{[1,k-1]}\in \mathfrak{H}$.
When $z\in D_{+}$, the result follows from the definition of $D_{+}$.
Let $z=\sum _{l=1}^{r}a_{l}\omega_{l}\,\,(0< a_{1},\cdots,a_{r}< 1)\in D$, then
\begin{align}
z_{k}+e^{{\pi}i}\vert \boldsymbol{\omega}_{\,k}\vert^{\,+}_{[1,k-1]}
&=\sum _{l=1}^{r}a_{l}\omega_{lk}+e^{{\pi}i}\vert \boldsymbol{\omega}_{\,k}\vert^{\,+}_{[1,k-1]}\nonumber \\
&=\sum _{l=1}^{k-1}(1-a_{l})e^{{\pi}i}\omega_{lk}+a_{k}+\sum _{l=k+1}^{r}a_{l}\omega_{lk}.\nonumber
\end{align}
Since \,$1-a_{l}>0$ for all $l=1,\cdots,k-1$, we obtain (\ref{eq:f_{+}uhp}).\\
(2) : 
By the condition {\bf[{ORC}]}, we have $\vert \boldsymbol{\omega}_{\,k}\vert^{\,+}_{[k+1,r]}\in \mathfrak{H}$.
Hence, $z\in D_{-}$, the result is proved by a similar argument as in (1).
If $z\in D$, then
\begin{align}
e^{{\pi}i}z_{k}+\vert {\boldsymbol{\omega}}_{\,k}\vert^{\,+}_{[k+1,r]}
&=\sum _{l=1}^{r}a_{l}e^{{\pi}i}\omega_{lk}+\vert {\boldsymbol{\omega}}_{\,k}\vert^{\,+}_{[k+1,r]}\nonumber \\
&=\sum _{l=1}^{k-1}a_{l}e^{{\pi}i}\omega_{lk}+e^{{\pi}i}a_{k}+\sum _{l=k+1}^{r}(1-a_{l})\omega_{lk}.\nonumber
\end{align}
Since \,$1-a_{l}>0$ for all $l=k+1,\cdots,r$, we obtain (\ref{eq:f_{-}uhp}).
\end{proof}

{\bf{Proof of Theorem\,4.11}.}
Suppose that $z\in D\cup D_{+}$. Then we have 
\begin{align}
f_{+}(s,z\mid {\boldsymbol{\omega}})
&=\sum _{k=0}^{r-1}(-1)^{k}\{\zeta_{r}(s,\vert {\boldsymbol{\omega}}\vert^{\,+}_{[1,k]}+e^{-{\pi}i}z\mid {\boldsymbol{\omega}}^{-}[k+1,r])\nonumber \\
{} & \quad+\zeta_{r}(s,\vert {\boldsymbol{\omega}}\vert^{\,+}_{[1,k+1]}+e^{-{\pi}i}z\mid {\boldsymbol{\omega}}^{-}[k+2,r])\}\nonumber \\
&=\sum _{k=1}^{r}(-1)^{k-1}\xi_{r}(s,\vert {\boldsymbol{\omega}}\vert^{\,+}_{[1,k-1]}+e^{-{\pi}i}z\mid \omega_{k};\widehat{{\boldsymbol{\omega}}}^{-}[k,r](k))\nonumber \\
&=\sum _{k=1}^{r}(-1)^{k-1}\left(\frac{e^{{\pi}i}}{\omega_{k}}\right)^s
\xi_{r}(s,z_{k}+e^{{\pi}i}\vert {\boldsymbol{\omega}}_{\,k}\vert^{\,+}_{[1,k-1]}\mid e^{{\pi}i};e^{{\pi}i}\widehat{{\boldsymbol{\omega}}}^{\,-}_{\,k}[k,r])\nonumber \\
&=\sum _{k=1}^{r}(-1)^{k-1}\left(\frac{e^{{\pi}i}}{\omega_{k}}\right)^s
\xi_{r}(s,z_{k}+e^{{\pi}i}\vert \boldsymbol{\omega}_{\,k}\vert^{\,+}_{[1,k-1]}\mid e^{{\pi}i};\widehat{\boldsymbol{\omega}}^{\,-}_{\,k}[1,k-1]).\nonumber
\end{align}
The second equality follows from (\ref{eq:def bilateral}) and the third one from Lemma\,\ref{prop:bilateral multiplication}.
We remark that all $r$-ple zeta functions and bilateral $r$-ple zeta functions appeared above 
are well-defined, since all parameters appeared in the above equalities satisfy the condition {\bf[{SOC}]}.

In addition, by $z\in D\cup D_{+}$, we may apply Lemma\,\ref{lem:f_{+}f_{-}uhp} to (\ref{eq:f_{+}xi}).
Hence it follows that all variables of the bilateral $r$-ple zeta functions appeared in (\ref{eq:f_{+}xi}) are in the upper half plane.
Therefore, by Theorem\,\ref{thm:multiple basic formula},
\begin{align}
f_{+}(s,z\mid {\boldsymbol{\omega}})
&=\frac{e^{\frac{\pi}{2}is}(2\pi)^s}{\Gamma(s)}\sum _{k=1}^{r}(-1)^{k-1}\omega_{k}^{-s}
\sum _{n=1}^{\infty }n^{s-1}e^{2{\pi}in(z_{k}+e^{{\pi}i}\vert {\boldsymbol{\omega}}_{\,k}\vert^{\,+}_{[1,k-1]})}\nonumber \\
{} & \quad \cdot\! 
\prod _{j=1}^{k-1}(1-e^{-2{\pi}in\omega_{jk}})^{-1}\prod _{j=k+1}^{r}(1-e^{2{\pi}in\omega_{jk}})^{-1}\nonumber \\
&=\frac{e^{\frac{\pi}{2}is}(2\pi)^s}{\Gamma(s)}\sum _{k=1}^{r}\omega_{k}^{-s}\sum _{n=1}^{\infty }
n^{s-1}e^{2{\pi}inz_{k}}\!\!\prod _{j=1,\,j\neq k}^{r}\!\!(1-e^{2{\pi}in\omega_{jk}})^{-1}.\nonumber
\end{align}
\indent
Let $z\in D\cup D_{-}$. Similarly to the discussion above, one observes
\begin{align}
f_{-}(s,z\mid {\boldsymbol{\omega}})
&=\sum _{k=0}^{r-1}(-1)^{k}\{\zeta_{r}(s,z+e^{-{\pi}i}\vert {\boldsymbol{\omega}}\vert^{\,+}_{[r+1-k,r]}\mid {\boldsymbol{\omega}}^{-}[r+1-k,r]) \nonumber \\
{} & \quad+\zeta_{r}(s,z+e^{-{\pi}i}\vert {\boldsymbol{\omega}}\vert^{\,+}_{[r-k,r]}\mid {\boldsymbol{\omega}}^{-}[r-k,r])\}\nonumber \\
&=\sum _{k=1}^{r}(-1)^{r-k}\xi_{r}(s,z+e^{-{\pi}i}\vert {\boldsymbol{\omega}}\vert^{\,+}_{[k+1,r]}\mid \omega_{k};\widehat{{\boldsymbol{\omega}}}^{-}[k,r](k))\nonumber \\
&=\sum _{k=1}^{r}(-1)^{r-k}\left(\frac{e^{{\pi}i}}{\omega_{k}}\right)^s
\xi_{r}(s,e^{{\pi}i}z_{k}+\vert {\boldsymbol{\omega}}_{\,k}\vert^{\,+}_{[k+1,r]}\mid e^{{\pi}i};\widehat{{\boldsymbol{\omega}}}^{\,-}_{\,k}[1,k-1])\nonumber \\
&=\frac{e^{\frac{\pi}{2}is}(2\pi)^s}{\Gamma(s)}\sum _{k=1}^{r}\omega_{k}^{-s}\sum _{n=1}^{\infty }
n^{s-1}e^{-2{\pi}inz_{k}}\!\!\prod _{j=1,\,j\neq k}^{r}\!\! (1-e^{-2{\pi}in\omega_{jk}})^{-1}.\nonumber
\end{align}
This completes the proof of Theorem \ref{thm:bilateral relation}.
\begin{chu}
Let $r=1$ and $\omega_{1}\in \mathfrak{H}$. 
Then, for all $s \in \mathbf{C}$, we have  
\begin{equation}
\label{eq:Lipschitz2}
f_{+}(s,z\mid \omega_{1})
=\xi_{1}(s,e^{-{\pi}i}z\mid \omega_{1})
=\frac{e^{\frac{\pi}{2}is}(2\pi)^s}{\Gamma(s)}\omega_{1}^{-s}\sum _{n=1}^{\infty }
n^{s-1}e^{2{\pi}inz_{1}} \,\,\,\,(z\in D_{+}),
\end{equation}
and 
\begin{equation}
\label{eq:Lipschitz3}
f_{-}(s,z\mid \omega_{1})
=\xi_{1}(s,z\mid \omega_{1})
=\frac{e^{\frac{\pi}{2}is}(2\pi)^s}{\Gamma(s)}\omega_{1}^{-s}\sum _{n=1}^{\infty }
n^{s-1}e^{-2{\pi}inz_{1}}  \,\,\,\,(z\in D_{-}).
\end{equation}
In particular, when $\Re(s)<0$, we have (\ref{eq:Lipschitz2}) and (\ref{eq:Lipschitz3}) under the condition $z\in D$.
\end{chu}

\begin{lem}
\label{prop:ext 2}
Assume that $z\in D\cup D_{\pm }$. If $\arg(\omega_{j})\neq \arg(\omega_{k})\,\,(j\neq k)$, 
then the right-hand sides of {\rm(\ref{eq:f_{+}basic})} and {\rm(\ref{eq:f_{-}basic})} converges absolutely.
\end{lem}
\begin{proof}
If $\arg(\omega_{j})\neq \arg(\omega_{k})\,\,(j\neq k)$, we may change the order of parameters such as  
$\arg(\omega_{j})<\arg(\omega_{k})\,\,(j<k)$. 
Hence, by Lemma\,\ref{lem:f_{+}f_{-}uhp}, we obtain the assertion.
\end{proof}

\begin{cor}
\label{cor:Narukawa basic0}
Let $r\geq 2$, $z\in D$. Suppose that $z,\omega_{1},\cdots,\omega_{r}\in \mathfrak{H}$ satisfy the condition {\bf[{ORC}]}.
If $z\in D\cup D_{\pm}$, one has 
\begin{align}
\label{eq:f_{+} D1}
\frac{\partial f_{\pm}}{\partial s}(1-m,z\mid {\boldsymbol{\omega}})
&=\frac{(-1)^{m-1}(m-1)!}{(2{\pi}i)^{m-1}}\sum _{k=1}^{r}\omega_{k}^{m-1}\sum _{n=1}^{\infty }
\frac{e^{\pm 2{\pi}inz_{k}}}{n^m}\!\!\prod _{j=1,\,j\neq k}^{r}\!\!(1-e^{\pm2{\pi}in\omega_{jk}})^{-1}. 
\end{align}
In particular, for any $z\in \mathbf{C}$,
\begin{equation}
\label{eq:f_{+} D2}
\exp\left(-\frac{\partial f_{\pm }}{\partial s}(0,z\mid {\boldsymbol{\omega}})\right)
=\prod _{k=1}^{r}(\widetilde{x_{k}^{\pm 1};\,\widehat{{\boldsymbol{q}}}_{\,k}^{\pm 1}})_{r-1,\infty }.
\end{equation}
\end{cor}
\begin{proof}
By (\ref{eq:f_{+}xi}) and Corollary \ref{cor:bilateral zero}, one has 
\begin{align}
\frac{\partial f_{\pm }}{\partial s}
(1-m,z\mid {\boldsymbol{\omega}})
&=\lim_{s\to 0}\frac{f_{\pm }(s+1-m,z\mid {\boldsymbol{\omega}})-f_{\pm }(1-m,z\mid {\boldsymbol{\omega}})}{s}\nonumber \\
&=\frac{(-1)^{m-1}(m-1)!}{(2{\pi}i)^{m-1}}
\sum _{k=1}^{r}\omega_{k}^{m-1}\sum _{n=1}^{\infty }
\frac{e^{\pm 2{\pi}inz_{k}}}{n^m}\!\!\prod _{j=1,\,j\neq k}^{r}\!\!(1-e^{\pm2{\pi}in\omega_{jk}})^{-1}.\nonumber
\end{align}

Similarly, one obtains
\begin{align}
\exp\left(-\frac{\partial f_{+}}{\partial s}(0,z\mid {\boldsymbol{\omega}})\right)
&=\prod _{k=1}^{r}\exp\left(-(-1)^{k-1}\frac{\partial \xi_{r}}{\partial s}
(0,z_{k}+e^{{\pi}i}\vert {\boldsymbol{\omega}}_{\,k}\vert^{\,+}_{[1,k-1]}\mid e^{{\pi}i};\widehat{{\boldsymbol{\omega}}}^{\,-}_{\,k}[1,k-1])\right)\nonumber \\
&=\prod _{k=1}^{r}{(\vert {\boldsymbol{q}}^{\,-1}_{\,k}\vert^{\times}_{[1,k-1]}x_{k};\,\widehat{{\boldsymbol{q}}}^{\,-1}_{\,k}[1,k-1])^{(-1)^{k-1}}_{r-1,\infty }}\nonumber \\
&=\prod _{k=1}^{r}(\widetilde{x_{k};\,\widehat{{\boldsymbol{q}}}_{\,k}})_{r-1,\infty }.\nonumber
\end{align}
The second equality follows from (\ref{eq:bilateral q-fact}) of Corollary \ref{cor:q-fact} 
and the third one from (\ref{eq:general q-fact+}).\\

Similarly,
\begin{align}
\exp\left(-\frac{\partial f_{-}}{\partial s}(0,z\mid {\boldsymbol{\omega}})\right)
&=\prod _{k=1}^{r}\exp\left(-(-1)^{r-k}\frac{\partial \xi_{r}}{\partial s}
(0,e^{{\pi}i}z_{k}+\vert {\boldsymbol{\omega}}_{\,k}\vert^{\,+}_{[k+1,r]}\mid e^{{\pi}i};\widehat{{\boldsymbol{\omega}}}^{\,-}_{\,k}[1,k-1]\right)\nonumber \\
&=\prod _{k=1}^{r}{(\vert {\boldsymbol{q}}_{\,k}\vert^{\times}_{[k+1,r]}x_{k}^{-1};\,\widehat{{\boldsymbol{q}}}^{\,-1}_{\,k}[1,k-1])^{(-1)^{r-k}}_{r-1,\infty }}\nonumber \\
&=\prod _{k=1}^{r}(\widetilde{x_{k}^{-1};\,\widehat{{\boldsymbol{q}}}^{\,-1}_{\,k}})_{r-1,\infty }.\nonumber
\end{align}
\end{proof}

\begin{chu}
\label{chu:ext 3}
By Lemma\,\ref{prop:ext 2}, if $z\in D$ and $\arg(\omega_{j})\neq \arg(\omega_{k})\,(j\neq k)$, 
then the right-hand sides of (\ref{eq:f_{+} D1}) and (\ref{eq:f_{+} D2}) both converge absolutely.
\end{chu}

\section{Applications}

\subsection{Dedekind's $\eta$-inversion formula and Ramanujan's formula}
In this subsection, we assume that $\tau \in \mathfrak{H}$.
\begin{lem}
{\rm(1)} For any $s\in \mathbf{C}$, 
\begin{equation}
\label{eq:double zeta lem}
\zeta_{2}(s,\omega_{1}\mid \omega_{1},\omega_{2})-\zeta_{2}(s,\omega_{2}\mid \omega_{1},\omega_{2})
=(\omega_{1}^{-s}-\omega_{2}^{-s})\zeta(s).
\end{equation}
Here $\zeta(s)$ is the Riemann zeta function.\\
{\rm(2)} For any $N \in \mathbf{N}$, 
\begin{align}
\label{eq:double zeta pole lem}
\lim_{s\to 2N}(1-e^{{\pi}is})\zeta_{2}(s,z \mid \omega_{1},\omega_{2})=-\frac{\pi{i}}{\omega_{1}\omega_{2}}\delta_{1,N}.
\end{align}
Here $\delta_{1,N}$ is the Kronecker's delta.
\end{lem}
\begin{proof}
{\rm(1) :} 
By Lemma \ref{prop:Barnes zeta}, we obtain
\begin{align}
(LHS)&=\lim_{z\to 0}\{\zeta_{2}(s,z+\omega_{1}\mid \omega_{1},\omega_{2})-\zeta_{2}(s,z+\omega_{2}\mid \omega_{1},\omega_{2})\} \nonumber \\
&=\lim_{z\to 0}\{\{\zeta_{2}(s,z\mid \omega_{1},\omega_{2})-\zeta_{1}(s,z\mid \omega_{2})\}
-\{\zeta_{2}(s,z\mid \omega_{1},\omega_{2})-\zeta_{1}(s,z\mid \omega_{1})\}\} \nonumber \\
&=\lim_{z\to 0}\{\{\zeta_{1}(s,z\mid \omega_{1})-z^{-s}\}-\{\zeta_{1}(s,z\mid \omega_{2})-z^{-s}\}\} \nonumber \\
&=\lim_{z\to 0}\{\zeta_{1}(s,z+\omega_{1}\mid \omega_{1})-\zeta_{1}(s,z+\omega_{2}\mid \omega_{2})\} \nonumber \\
&=\zeta_{1}(s,\omega_{1}\mid \omega_{1})-\zeta_{1}(s,\omega_{2}\mid \omega_{2})
=(\omega_{1}^{-s}-\omega_{2}^{-s})\zeta(s). \nonumber
\end{align}
{\rm(2) :} 
By (\ref{eq:B_{r,0}}) and (\ref{eq:BarnesSV2}), we obtain
\begin{align}
(LHS)&=\lim_{s\to 2N}\frac{1-e^{{\pi}is}}{s-2N}(s-2N)\zeta_{2}(s,z \mid \omega_{1},\omega_{2}) \nonumber \\
&=-\pi{i}\mathop{\rm Res}_{s=2N}\zeta_{2}(s,z \mid \omega_{1},\omega_{2})ds
=-\pi{i}B_{2,0}(z \mid \omega_{1},\omega_{2})\delta_{1,N}
=-\frac{\pi{i}}{\omega_{1}\omega_{2}}\delta_{1,N}.\nonumber
\end{align}
\end{proof}
\begin{lem}
\label{prop:eta basic}
Define
\begin{align}
g(s,\tau):=&
\xi_{2}(s,\tau\mid e^{{\pi}i};\tau)-\left(e^{{\pi}i}\frac{1}{\tau}\right)^{s}\xi_{2}\left(s,e^{{\pi}i}\frac{1}{\tau}\,\bigg\vert \, e^{{\pi}i};e^{{\pi}i}\frac{1}{\tau}\right) \nonumber \\
=&\xi_{2}(s,\tau\mid e^{{\pi}i};\tau)-\xi_{2}(s,1\mid \tau;1). \nonumber
\end{align}
{\rm(1)} We have
\begin{equation}
\label{eq:eta basic1}
\frac{\partial g}{\partial s}(0,\tau)
=-\frac{{\pi}i}{4}+\frac{{\pi}i}{12}\left(\tau+\frac{1}{\tau}\right)+\frac{1}{2}\log{\tau}.
\end{equation}
{\rm(2)} For any $N\in \mathbf{N}$,
\begin{align}
\label{eq:Ramanujan basic1}
\frac{\partial g}{\partial s}(-2N,\tau)
=\frac{{\pi}iB_{2,2+2N}(0\mid 1,\tau)}{(2N+2)(2N+1)}+\frac{(-1)^N}{2}(\tau^{2N}-1)(2N)!(2\pi)^{-2N}\zeta(2N+1).
\end{align}
{\rm(3)} For any $ N\in \mathbf{N}$,
\begin{equation}
\label{eq:Eisenstein basic1}
g(2N,\tau)=(\tau^{-2N}-1)\left(-\frac{1}{2}\frac{B_{2N}}{(2N)!}(2{\pi}i)^{2N}\right)+\frac{\pi{i}}{\tau}\delta_{1,N}.
\end{equation}
\end{lem}
\begin{proof}
By (\ref{eq:def bilateral}), we observe that
\begin{align}
g(s,\tau)&=\{\zeta_{2}(s,1+\tau\mid 1,\tau)+\zeta_{2}(s,\tau\mid e^{{\pi}i},\tau)\}
-\{\zeta_{2}(s,1+\tau\mid \tau,1)+\zeta_{2}(s,1\mid e^{-{\pi}i}\tau,1)\} \nonumber \\
&=\zeta_{2}(s,\tau\mid e^{{\pi}i},\tau)-\zeta_{2}(s,1\mid e^{-{\pi}i}\tau,1)
=\zeta_{2}(s,\tau\mid e^{{\pi}i},\tau)-e^{{\pi}is}\zeta_{2}(s,e^{{\pi}i}\mid \tau,e^{{\pi}i}). \nonumber
\end{align}
{\rm(1) :} 
Using the above relation, we have
\begin{equation}
\frac{\partial g}{\partial s}(0,\tau)
=-{\pi}i\zeta_{2}(0,e^{{\pi}i}\mid \tau,e^{{\pi}i})+\left.\frac{\partial }{\partial s}
\{\zeta_{2}(s,\tau\mid e^{{\pi}i},\tau)-\zeta_{2}(s,e^{{\pi}i}\mid \tau,e^{{\pi}i})\}\right|_{s=0}.\nonumber
\end{equation}
By (\ref{eq:BarnesSV}) and (\ref{eq:B_{2,2}})
\begin{equation}
-{\pi}i\zeta_{2}(0,e^{{\pi}i}\mid \tau,e^{{\pi}i})
=-{\pi}i\frac{B_{2,2}(e^{{\pi}i}\mid \tau,e^{{\pi}i})}{2!}
=\frac{{\pi}i}{4}+\frac{{\pi}i}{12}\left(\tau+\frac{1}{\tau}\right).\nonumber
\end{equation}
On the other hand, by (\ref{eq:double zeta lem}) and the fact $\zeta(0)=-\frac{1}{2}$, one finds 
\begin{align}
&\left.\frac{\partial }{\partial s}
\{\zeta_{2}(s,\tau\mid e^{{\pi}i},\tau)-\zeta_{2}(s,e^{{\pi}i}\mid \tau,e^{{\pi}i})\}\right|_{s=0}\nonumber \\
&=\left.\frac{\partial }{\partial s}\{(\tau^{-s}-e^{-{\pi}is})\zeta(s)\}\right|_{s=0}
=(-\log{\tau}+{\pi}i)\zeta(0)
=\frac{1}{2}\log{\tau}-\frac{{\pi}i}{2}.\nonumber
\end{align}
Consequently we obtain
\begin{equation}
\frac{\partial g}{\partial s}(0,\tau)
=-\frac{{\pi}i}{4}+\frac{{\pi}i}{12}\left(\tau+\frac{1}{\tau}\right)+\frac{1}{2}\log{\tau}.\nonumber
\end{equation}
{\rm(2) :} 
Similarly, we have 
\begin{align}
\frac{\partial g}{\partial s}(-2N,\tau)
=-{\pi}i\zeta_{2}(-2N,e^{{\pi}i}\mid \tau,e^{{\pi}i})+\!\!\left.\frac{\partial }{\partial s}\{\zeta_{2}(s,\tau\mid e^{{\pi}i},\tau)-\zeta_{2}(s,e^{{\pi}i}\mid \tau,e^{{\pi}i})\}\right|_{s=-2N}.\nonumber
\end{align}
By (\ref{eq:BarnesSV}) and Lemma \ref{prop:multiple Brnoulli1}, we have 
\begin{equation}
-{\pi}i\zeta_{2}(-2N,e^{{\pi}i}\mid \tau,e^{{\pi}i})
=-{\pi}i\frac{B_{2,2+2N}(e^{{\pi}i}\mid \tau,e^{{\pi}i})}{(2N+2)(2N+1)}
=\frac{{\pi}iB_{2,2+2N}(0\mid 1,\tau)}{(2N+2)(2N+1)}.\nonumber
\end{equation}
Also, by the fact $\zeta(-2N)=0$ for any $N\in \mathbf{N}$, we have
\begin{align}
\left.\frac{\partial }{\partial s}
\{\zeta_{2}(s,\tau\mid e^{{\pi}i},\tau)-\zeta_{2}(s,e^{{\pi}i}\mid \tau,e^{{\pi}i})\}\right|_{s=-2N}
&=\left.\frac{\partial }{\partial s}\{(\tau^{-s}-e^{-{\pi}is})\zeta(s)\}\right|_{s=-2N}\nonumber \\
&=(\tau^{2N}-1)\frac{\partial \zeta}{\partial s}(-2N).\nonumber
\end{align}
Recall the functional equation of the Riemann zeta function.
\begin{equation}
\zeta(s)=2\Gamma(1-s)\sin\left(\frac{\pi{s}}{2}\right)(2\pi)^{s-1}\zeta(1-s).\nonumber
\end{equation}
Hence, it follows that 
\begin{equation}
\frac{\partial \zeta}{\partial s}(-2N)=\frac{(-1)^N}{2}(2N)!(2\pi)^{-2N}\zeta(2N+1).\nonumber
\end{equation}
Therefore,
\begin{align}
\frac{\partial g}{\partial s}(-2N,\tau)
=\frac{{\pi}iB_{2,2+2N}(0\mid 1,\tau)}{(2N+2)(2N+1)}+\frac{(-1)^N}{2}(\tau^{2N}-1)(2N)!(2\pi)^{-2N}\zeta(2N+1).\nonumber
\end{align}
{\rm(3) :}
We remark that 
\begin{align}
\zeta_{2}(s,\tau\mid e^{{\pi}i},\tau)-e^{{\pi}is}\zeta_{2}(s,e^{{\pi}i}\mid \tau,e^{{\pi}i})
&=\{\zeta_{2}(s,\tau\mid e^{{\pi}i},\tau)-\zeta_{2}(s,e^{{\pi}i}\mid \tau,e^{{\pi}i})\} \nonumber \\
{} & \quad+(1-e^{{\pi}is})\zeta_{2}(s,e^{{\pi}i}\mid \tau,e^{{\pi}i}). \nonumber
\end{align}
Therefore, by (\ref{eq:double zeta lem}), (\ref{eq:double zeta pole lem}) and $\zeta(2N)=-\frac{1}{2}\frac{B_{2N}}{(2N)!}(2{\pi}i)^{2N}$, we have
\begin{align}
g(2N,\tau)
&=\{\zeta_{2}(2N,\tau\mid e^{{\pi}i},\tau)-\zeta_{2}(2N,e^{{\pi}i}\mid \tau,e^{{\pi}i})\}
+\lim_{s\to 2N}(1-e^{{\pi}is})\zeta_{2}(s,e^{{\pi}i}\mid \tau,e^{{\pi}i}) \nonumber \\
&=(\tau^{-2N}-1)\left(-\frac{1}{2}\frac{B_{2N}}{(2N)!}(2{\pi}i)^{2N}\right)+\frac{\pi{i}}{\tau}\delta_{1,N}.\nonumber
\end{align}
\end{proof}

\begin{prop}
\label{prop:eta&Ramanujan}
{\rm(1)} (Inversion formula for the Dedekind $\eta$-function)
Let
\begin{equation}
\eta(\tau)
:=e^{\frac{{\pi}i}{12}\tau}\prod _{m=1}^{\infty }{(1-e^{2{\pi}im\tau})}
\end{equation}
be the Dedekind $\eta$-function. We have 
\begin{equation}
\label{eq:eta inversion}
\eta\left(e^{{\pi}i} \frac{1}{\tau}\right)
=\sqrt[]{\frac{\tau}{i}}
\,\eta(\tau).
\end{equation}
Here $\sqrt{\frac{\tau}{i}}$ being that branch taking the value $1$ at $\tau=i$.\\
{\rm{(2)}} (Ramanujan's formula) For any $N\in \mathbf{N}$, we have
\begin{align}
\label{eq:Ramanujan's formula}
\frac{1}{2}\zeta(2N+1)+\sum _{n=1}^{\infty }\frac{1}{n^{2N+1}}\frac{e^{2{\pi}in\tau}}{1-e^{2{\pi}in\tau}}
&=\tau^{2N}\left\{\frac{1}{2}\zeta(2N+1)+\sum _{n=1}^{\infty }\frac{1}{n^{2N+1}}\frac{e^{-2{\pi}in\frac{1}{\tau}}}{1-e^{-2{\pi}in\frac{1}{\tau}}}\right\} \\
{} & \quad+\frac{1}{2}\frac{(2\pi{i})^{2N+1}}{(2N+2)!}B_{2,2+2N}(0\mid \tau,1).\nonumber
\end{align}
{\rm(3)} (Inversion formula for the Eisenstein series/Lambert series) For any $N\in \mathbf{N}$, we have
\begin{equation}
\label{eq:Eisenstein inversion}
\sum _{n=1}^{\infty }\frac{n^{2N-1}e^{-2{\pi}in\frac{1}{\tau}}}{1-e^{-2{\pi}in\frac{1}{\tau}}}-\frac{B_{2N}}{4N}
=\tau^{2N}\left(\sum _{n=1}^{\infty }\frac{n^{2N-1}e^{2{\pi}in\tau}}{1-e^{2{\pi}in\tau}}-\frac{B_{2N}}{4N}\right)-\frac{\tau}{4\pi{i}}\delta_{1,N}.
\end{equation}
\end{prop}
\begin{proof}
{\rm(1) :}
By Corollary \ref{cor:bilateral zero} and Corollary \ref{cor:q-fact}, we have
\begin{align}
\exp\left(\frac{\partial g}{\partial s}(0,\tau)\right)
=\exp\left(\frac{\partial \xi_{2}}{\partial s}(0,\tau\mid e^{{\pi}i};\tau)
-\frac{\partial \xi_{2}}{\partial s}\left(0,e^{{\pi}i}\frac{1}{\tau}\,\bigg\vert \, e^{{\pi}i};e^{{\pi}i}\frac{1}{\tau}\right)\right)
=\frac{h\left(e^{{\pi}i}\frac{1}{\tau}\right)}{h(\tau)}.\nonumber
\end{align}
Here, $h(\tau):=\prod _{m=1}^{\infty }{(1-e^{2{\pi}im\tau})}$. 
On the other hand, by (\ref{eq:eta basic1})
\begin{equation}
\exp\left(\frac{\partial g}{\partial s}(0,\tau)\right)
=\exp\left(-\frac{{\pi}i}{4}+\frac{{\pi}i}{12}\left(\tau+\frac{1}{\tau}\right)+\frac{1}{2}\log{\tau}\right).\nonumber
\end{equation}
Therefore,
\begin{equation}
e^{\frac{{\pi}i}{12}\frac{e^{{\pi}i}}{\tau}}h\left(e^{{\pi}i}\frac{1}{\tau}\right)
=\exp\left(\frac{1}{2}\log{\tau}-\frac{{\pi}i}{4}\right)e^{\frac{{\pi}i}{12}\tau}h(\tau).\nonumber
\end{equation}
Since $\eta(\tau)=e^{\frac{{\pi}i}{12}\tau}h(\tau)$, we obtain (\ref{eq:eta inversion}).\\
{\rm(2) :}
By Corollary \ref{cor:bilateral zero} and Corollary \ref{cor:q-fact}, we have
\begin{align}
\frac{\partial g}{\partial s}(-2N,\tau)
&=\frac{(2N)!}{(2{\pi}i)^{2N}}\sum _{n=1}^{\infty }\frac{1}{n^{2N+1}}\frac{e^{2{\pi}in\tau}}{1-e^{2{\pi}in\tau}}
-\tau^{2N}\frac{(2N)!}{(2{\pi}i)^{2N}}\sum _{n=1}^{\infty }\frac{1}{n^{2N+1}}\frac{e^{-2{\pi}in\frac{1}{\tau}}}{1-e^{-2{\pi}in\frac{1}{\tau}}}.\nonumber
\end{align}
Hence we obtain (\ref{eq:Ramanujan's formula}) by (\ref{eq:Ramanujan basic1}).\\
{\rm(3) :}
By (\ref{eq:multiple basic formula}), we have
\begin{align}
g(2N,\tau)
&=\frac{(2{\pi}i)^{2N}}{(2N-1)!}\sum _{n=1}^{\infty }\frac{n^{2N-1}e^{2{\pi}in\tau}}{1-e^{2{\pi}in\tau}}
-\tau^{-2N}\frac{(2{\pi}i)^{2N}}{(2N-1)!}\sum _{n=1}^{\infty }\frac{n^{2N-1}e^{-2{\pi}in\frac{1}{\tau}}}{1-e^{-2{\pi}in\frac{1}{\tau}}}.\nonumber
\end{align}
Hence we obtain (\ref{eq:Eisenstein inversion}) by (\ref{eq:Eisenstein basic1}).
\end{proof}
\subsection{Fourier expansion of the Barnes zeta function}
\begin{thm}
\label{thm:Barnes zeta F.E}
Let $r\geq 2$, $z\in D$. We assume that $z,\omega_{1},\cdots,\omega_{r}\in \mathfrak{H}$ satisfy the condition {\bf[{ORC}]}.
Then, for all $s\in \mathbf{C}$, we obtain
\begin{align}
\label{eq:Barnes & Bilateral}
\zeta_{r}(s,z\mid {\boldsymbol{\omega}})
&=\frac{1}{2i\sin({\pi}s)} \\
{} & \quad \cdot\! \left\{
\sum _{k=1}^{r}(-1)^{k-1}\xi_{r}(s,\vert {\boldsymbol{\omega}}\vert^{\,+}_{[1,k-1]}+e^{-{\pi}i}z\mid \omega_{k};\widehat{{\boldsymbol{\omega}}}^{-}[k,r](k))\right. \nonumber \\
{} & \quad \left.-\sum _{k=1}^{r}(-1)^{r-k}e^{-{\pi}is}\xi_{r}(s,z+e^{-{\pi}i}\vert {\boldsymbol{\omega}}\vert^{\,+}_{[k+1,r]}\mid \omega_{k};\widehat{{\boldsymbol{\omega}}}^{-}[k,r](k))\right\} \nonumber \\
\label{eq:Barnes F.E}
&=(2\pi)^{s-1}\Gamma(1-s) \\
{} & \quad \cdot\! \left\{e^{\frac{\pi}{2}i(s-1)}\sum _{k=1}^{r}\omega_{k}^{-s}\sum _{n=1}^{\infty }
n^{s-1}e^{2{\pi}inz_{k}}\!\!\prod _{j=1,\,j\neq k}^{r}\!\!(1-e^{2{\pi}in\omega_{jk}})^{-1}\right. \nonumber \\
{} & \quad \left.+e^{-\frac{\pi}{2}i(s-1)}\sum _{k=1}^{r}\omega_{k}^{-s}\sum _{n=1}^{\infty }
n^{s-1}e^{-2{\pi}inz_{k}}\!\!\prod _{j=1,\,j\neq k}^{r}\!\!(1-e^{-2{\pi}in\omega_{jk}})^{-1}\right\}.\nonumber
\end{align}
\end{thm}
\begin{proof}
Put
\begin{equation}
F(s,z\mid {\boldsymbol{\omega}}):=f_{+}(s,z\mid {\boldsymbol{\omega}})-e^{-\pi{i}s}f_{-}(s,z\mid {\boldsymbol{\omega}}).\nonumber
\end{equation}
Let $z\in D$. Then, by Theorem \ref{thm:bilateral relation}, we have 
\begin{align}
F(s,z\mid \boldsymbol{\omega})
&=\sum _{k=1}^{r}(-1)^{k-1}\xi_{r}(s,\vert {\boldsymbol{\omega}}\vert^{\,+}_{[1,k-1]}+e^{-{\pi}i}z\mid \omega_{k};\widehat{{\boldsymbol{\omega}}}^{-}[k,r](k)) \nonumber \\
{} & \quad -\sum _{k=1}^{r}(-1)^{r-k}e^{-{\pi}is}\xi_{r}(s,z+e^{-{\pi}i}\vert {\boldsymbol{\omega}}\vert^{\,+}_{[k+1,r]}\mid \omega_{k};\widehat{{\boldsymbol{\omega}}}^{-}[k,r](k)) \nonumber \\
&=\frac{(2\pi)^s}{\Gamma(s)}
\left\{e^{\frac{\pi}{2}is}\sum _{k=1}^{r}\omega_{k}^{-s}\sum _{n=1}^{\infty }
n^{s-1}e^{2{\pi}inz_{k}}\!\!\prod _{j=1,\,j\neq k}^{r}\!\!(1-e^{2{\pi}in\omega_{jk}})^{-1}\right.\nonumber \\
{} & \quad \left.-e^{-\frac{\pi}{2}is}\sum _{k=1}^{r}\omega_{k}^{-s}\sum _{n=1}^{\infty }
n^{s-1}e^{-2{\pi}inz_{k}}\!\!\prod _{j=1,\,j\neq k}^{r}\!\!(1-e^{-2{\pi}in\omega_{jk}})^{-1}\right\}.\nonumber
\end{align}
On the other hand, by the definitions of $f_{+}(s,z\mid {\boldsymbol{\omega}})$ and $f_{-}(s,z\mid {\boldsymbol{\omega}})$,
\begin{align}
F(s,z\mid {\boldsymbol{\omega}})
&=(e^{{\pi}is}-e^{-{\pi}is})\zeta_{r}(s,z\mid {\boldsymbol{\omega}})\nonumber \\
{} & \quad+(-1)^{r-1}(\zeta_{r}(s,\vert {\boldsymbol{\omega}}\vert^{+}+e^{-{\pi}i}z\mid {\boldsymbol{\omega}})
-e^{-{\pi}is}\zeta_{r}(s, z+e^{-{\pi}i}\vert {\boldsymbol{\omega}}\vert^{+} \mid e^{-{\pi}i}{\boldsymbol{\omega}})).\nonumber
\end{align}
Since $z\in D$, one has $e^{{\pi}i}(z+e^{-{\pi}i}\vert {\boldsymbol{\omega}}\vert^{+})= \vert{\boldsymbol{\omega}}\vert^{+}+e^{-{\pi}i}z$. 
By Lemma \ref{prop:multiplication} we have 
\begin{align}
e^{-{\pi}is}\zeta_{r}(s, z+e^{-{\pi}i}\vert {\boldsymbol{\omega}}\vert^{+} \mid e^{-{\pi}i}{\boldsymbol{\omega}})
=&\zeta_{r}(s, e^{{\pi}i}(z+e^{-{\pi}i}\vert {\boldsymbol{\omega}}\vert^{+}) \mid {\boldsymbol{\omega}})\nonumber \\
=&\zeta_{r}(s,\vert {\boldsymbol{\omega}}\vert^{+}+e^{-{\pi}i}z\mid {\boldsymbol{\omega}}).\nonumber
\end{align}
Therefore,
\begin{equation}
F(s,z\mid {\boldsymbol{\omega}})
=(e^{{\pi}is}-e^{-{\pi}is})\zeta_{r}(s,z\mid {\boldsymbol{\omega}})
=2i\sin({\pi}s)\zeta_{r}(s,z\mid {\boldsymbol{\omega}}).\nonumber
\end{equation}
Hence, we have (\ref{eq:Barnes & Bilateral}). Moreover it is obvious from Theorem \ref{thm:bilateral relation} that  
\begin{align}
\zeta_{r}(s,z\mid {\boldsymbol{\omega}})
&=\frac{(2\pi)^s}{2i\Gamma(s)\sin({\pi}s)}
\left\{e^{\frac{\pi}{2}is}\sum _{k=1}^{r}\omega_{k}^{-s}\sum _{n=1}^{\infty }
n^{s-1}e^{2{\pi}inz_{k}}\!\!\prod _{j=1,\,j\neq k}^{r}\!\!(1-e^{2{\pi}in\omega_{jk}})^{-1}\right.\nonumber \\
{} & \quad \left.-e^{-\frac{\pi}{2}is}\sum _{k=1}^{r}\omega_{k}^{-s}\sum _{n=1}^{\infty }
n^{s-1}e^{-2{\pi}inz_{k}}\!\!\prod _{j=1,\,j\neq k}^{r}\!\!(1-e^{-2{\pi}in\omega_{jk}})^{-1}\right\}.\nonumber
\end{align}
Hence (\ref{eq:Barnes F.E}) follows from the reflection formula for the Gamma function.
\end{proof}
\begin{chu}
The formula (\ref{eq:Barnes F.E}) has given by Komuri, Matsumoto and Tsumura \cite{KMT}. 
In particular, when $r=1$, $\Re(s)<0$ and $z=a\omega_{1}$ ($0<a<1$), the formula (\ref{eq:Barnes F.E}) gives the following functional equation of Hurwitz's zeta function.
\begin{align}
\zeta_{1}(s,z\mid \omega_{1})
&=\frac{1}{2i\sin({\pi}s)}
\left\{\xi_{1}(s,e^{-{\pi}i}z\mid \omega_{1})
-e^{-{\pi}is}\xi_{1}(s,z\mid \omega_{1})\right\} \\
&=(2\pi)^{s-1}\Gamma(1-s)\omega_{1}^{-s} \\
{} & \quad \cdot\!
\left\{e^{\frac{\pi}{2}i(s-1)}\sum _{n=1}^{\infty }n^{s-1}e^{2{\pi}inz_{1}}
+e^{-\frac{\pi}{2}i(s-1)}\sum _{n=1}^{\infty }n^{s-1}e^{-2{\pi}inz_{1}}\right\}.\nonumber
\end{align}
\end{chu}
\begin{cor}
We assume that $r\geq 2$, $z\in D$ and $\omega_{1},\cdots,\omega_{r}$ satisfy the condition {\bf[{ORC}]}. Then\\
{\rm(1)} For all $m\in \mathbf{N}_{0}$, 
\begin{align}
\label{eq:multiple Bernoulli F.E}
B_{r,m}(z\mid {\boldsymbol{\omega}})&=(-1)^{r}(2{\pi}i)^{r-1-m}m!\sum _{k=1}^{r}\omega_{k}^{m-r}
\!\!\sum _{n\in \mathbf{Z}\backslash \{0\}}\!\!n^{r-1-m}e^{2{\pi}inz_{k}}\!\!\prod _{j=1,\,j\neq k}^{r}\!\!(1-e^{2{\pi}in\omega_{jk}})^{-1}.
\end{align}
{\rm(2)} 
For all $m\in \mathbf{N}$, 
\begin{equation}
\label{eq:multiple Bernoulli F.E2}
\sum_{k=1}^{r}\omega_{k}^{-(r+m)}\sum _{n\in \mathbf{Z}\backslash \{0\}}
n^{r+m-1}e^{2{\pi}inz_{k}}\!\!\prod _{j=1,\,j\neq k}^{r}\!\!(1-e^{2{\pi}in\omega_{jk}})^{-1}=0.
\end{equation}
\end{cor}

\begin{proof}
{\rm(1) :}
We notice that the left-hand side of (\ref{eq:multiple Bernoulli F.E}) is a multiple Bernoulli polynomial, 
which is a rational function of $\omega_{1},\cdots,\omega_{r}$ 
and that the right-hand side of (\ref{eq:multiple Bernoulli F.E}) converges absolutely, 
when $z\in D$ and $\arg(\omega_{j})\neq \arg(\omega_{k})\,\,(j\neq k)$. 
Hence, by analytic continuation, it is enough to show the assertion when we assume that $\omega_{1},\cdots,\omega_{r}$ satisfy the same conditions of Theorem \ref{thm:bilateral relation}.

If $z\in D$, by (\ref{eq:Barnes F.E}), for all $m\in \mathbf{N}$, we have 
\begin{equation}
\zeta_{r}(1-m,z\mid {\boldsymbol{\omega}})
=\frac{(m-1)!}{(2{\pi}i)^{m}}\sum _{k=1}^{r}\omega_{k}^{m-1}\sum _{n\in \mathbf{Z}\backslash \{0\}}
\frac{e^{2{\pi}inz_{k}}}{n^m}\!\!\prod _{j=1,\,j\neq k}^{r}\!\!(1-e^{2{\pi}in\omega_{jk}})^{-1}.\nonumber
\end{equation}
Using (\ref{eq:BarnesSV}), 
we have 
\begin{equation}
B_{r,r+m-1}(z\mid {\boldsymbol{\omega}})=(-1)^{r}\frac{(m+r-1)!}{(2{\pi}i)^{m}}\sum _{k=1}^{r}\omega_{k}^{m-1}
\!\!\!\!\sum _{n\in \mathbf{Z}\backslash \{0\}}
\!\!\!\!\frac{e^{2{\pi}inz_{k}}}{n^m}\!\!\prod _{j=1,\,j\neq k}^{r}\!\!(1-e^{2{\pi}in\omega_{jk}})^{-1}.\nonumber
\end{equation}
Replacing $m$ with $m-r+1$, for all $m \in \mathbf{Z}_{\geq r}$, we obtain (\ref{eq:multiple Bernoulli F.E}). 
On the other hand, by (\ref{eq:Barnes F.E}), if $z\in D$, we see that 
\begin{align}
\mathop{\rm Res}_{s=m}\zeta_{r}(s,\,z\mid {\boldsymbol{\omega}})ds
&=\lim_{s\to{m}}(s-m)\zeta_{r}(s,\,z\mid {\boldsymbol{\omega}})\nonumber \\
&=\frac{(-1)^{m}(2\pi{i})^{m-1}}{(m-1)!}
\sum _{k=1}^{r}\omega_{k}^{-m}
\!\!\!\!\sum _{n\in \mathbf{Z}\backslash \{0\}}
\!\!\!\!n^{m-1}e^{2{\pi}inz_{k}}\!\!\prod _{j=1,\,j\neq k}^{r}\!\!(1-e^{2{\pi}in\omega_{jk}})^{-1}\nonumber
\end{align}
for all \,$m=1,\cdots,r$. 
Hence it follows from (\ref{eq:BarnesSV2}) that 
\begin{align}
B_{r,r-m}(z\mid {\boldsymbol{\omega}})
&=(-1)^{r}(2\pi{i})^{m-1}(r-m)!
\sum _{k=1}^{r}\omega_{k}^{-m}
\!\!\!\!\sum _{n\in \mathbf{Z}\backslash \{0\}}
\!\!\!\!n^{m-1}e^{2{\pi}inz_{k}}\!\!\prod _{j=1,\,j\neq k}^{r}\!\!(1-e^{2{\pi}in\omega_{jk}})^{-1}.\nonumber
\end{align}
Replacing $m$ with $r-m$, we obtain the result.\\
{\rm(2) :}
Since $\zeta_{r}(s,\,z\mid \boldsymbol{\omega})$ is holomorphic when $\Re(s)>r$, we have
\begin{equation}
\lim_{s\to{m}}\frac{\zeta_{r}(s+r,\,z\mid {\boldsymbol{\omega}})}
{\Gamma(1-r-s)}=0\,\,(m\in \mathbf{N}).\nonumber
\end{equation}
Since $\omega_{1},\cdots,\omega_{r}$ satisfy the same conditions of Theorem \ref{thm:bilateral relation}, 
if $z\in D$, it follows that 
\begin{align}
\lim_{s\to{m}}\frac{\zeta_{r}(s+r,\,z\mid {\boldsymbol{\omega}})}
{\Gamma(1-r-s)}
&=(2\pi{i})^{r+m-1}\sum _{k=1}^{r}\omega_{k}^{-(r+m)}
\!\!\!\!\sum _{n\in \mathbf{Z}\backslash \{0\}}
\!\!\!\!n^{r+m-1}e^{2{\pi}inz_{k}}
\!\!\prod _{j=1,\,j\neq k}^{r}\!\!(1-e^{2{\pi}in\omega_{jk}})^{-1}.\nonumber
\end{align}
Hence (\ref{eq:multiple Bernoulli F.E2}) follows immediately.
\end{proof}
\begin{chu}
Suppose $r=1$. If $z=a\omega_{1}$($0<a<1$), for all $m\in \mathbf{N}$, one has 
\begin{equation}
B_{1,m}(z\mid \omega_{1})=-\frac{m!}{(2{\pi}i)^{m}}\omega_{1}^{m-1}\sum _{n\in \mathbf{Z}\backslash \{0\}}
\frac{e^{2{\pi}inz_{1}}}{n^m}.\nonumber
\end{equation}
Here, when $m=1$, the summention on $n$ is meant 
\begin{equation}
\lim_{N\to \infty}\sum _{n=-N,\,n\neq 0}^{N}\frac{e^{2{\pi}inz_{1}}}{n}.\nonumber
\end{equation}
\end{chu}
\subsection{Multiple Iseki's formulas}
We prove the following multiple-analogue of Iseki's formula.
\begin{thm}
\label{thm:multiple Iseki}
We assume that $r\geq 2$, $z\in D$ and $\omega_{1},\cdots,\omega_{r}$ satisfy the condition {\bf[{ORC}]}.
Put
\begin{equation}
f(s,z\mid {\boldsymbol{\omega}}):=\zeta_{r}(s,z\mid {\boldsymbol{\omega}})
+(-1)^{r-1}\zeta_{r}(s,\vert {\boldsymbol{\omega}}\vert^{+}+e^{-{\pi}i}z\mid {\boldsymbol{\omega}}).
\end{equation}
Then, for all $N\in \mathbf{N}_{0}$, we have
\begin{align}
\label{eq:Narukawa basic1}
\frac{\partial f}{\partial s}(-2N,z\mid {\boldsymbol{\omega}})
&=(-1)^{r+1}\pi{i}\frac{(2N)!}{(2N+r)!}B_{r,r+2N}(z\mid {\boldsymbol{\omega}}) \\
{} & \quad
+\frac{(-1)^{N}(2N)!}{(2\pi)^{2N}}
\sum _{k=1}^{r}\omega_{k}^{2N}\sum _{n=1}^{\infty }
\frac{e^{2{\pi}inz_{k}}}{n^{2N+1}}\!\!\prod _{j=1,\,j\neq k}^{r}\!\!(1-e^{2{\pi}in\omega_{jk}})^{-1}\nonumber \\
&=(-1)^{r}\pi{i}\frac{(2N)!}{(2N+r)!}B_{r,r+2N}(z\mid {\boldsymbol{\omega}})\nonumber \\
{} & \quad
+\frac{(-1)^{N}(2N)!}{(2\pi)^{2N}}
\sum _{k=1}^{r}\omega_{k}^{2N}\sum _{n=1}^{\infty }
\frac{e^{-2{\pi}inz_{k}}}{n^{2N+1}}\!\!\prod _{j=1,\,j\neq k}^{r}\!\!(1-e^{-2{\pi}in\omega_{jk}})^{-1}.\nonumber
\end{align}
In particular, for all $z\in \mathbf{C}$ , we have
\begin{align}
\label{eq:Narukawa basic2}
\exp\left(-\frac{\partial f}{\partial s}(0,z\mid {\boldsymbol{\omega}})\right)
&=\exp{\left(\frac{(-1)^{r}\pi{i}}{r!}B_{r,r}(z\mid {\boldsymbol{\omega}})\right)}
\prod _{k=1}^{r}(\widetilde{x_{k};\,\widehat{{\boldsymbol{q}}}_{\,k}})_{r-1,\infty } \\
&=\exp{\left(\frac{(-1)^{r-1}\pi{i}}{r!}B_{r,r}(z\mid {\boldsymbol{\omega}})\right)}
\prod _{k=1}^{r}(\widetilde{x_{k}^{-1};\,\widehat{{\boldsymbol{q}}}^{\,-1}_{\,k}})_{r-1,\infty }.\nonumber
\end{align}
\end{thm}

\begin{proof}
If $z\in D$, then
\begin{equation}
f(s,z\mid {\boldsymbol{\omega}})
=e^{-{\pi}is}\zeta_{r}(s,e^{-{\pi}i}z\mid e^{-{\pi}i}{\boldsymbol{\omega}})
+(-1)^{r-1}\zeta_{r}(s,\vert {\boldsymbol{\omega}}\vert^{+}+e^{-{\pi}i}z \mid {\boldsymbol{\omega}}).\nonumber 
\end{equation}
Thus, by (\ref{eq:BarnesSV}), Lemma \ref{prop:multiple Brnoulli1} and (\ref{eq:f_{+} D1}),
\begin{align}
\frac{\partial f}{\partial s}(-2N,z\mid {\boldsymbol{\omega}})
&=-\pi{i}\zeta_{r}(-2N,e^{-{\pi}i}z\mid e^{-{\pi}i}{\boldsymbol{\omega}})
+\frac{\partial f_{+}}{\partial s}(-2N,z\mid {\boldsymbol{\omega}})\nonumber \\
&=(-1)^{r+1}\pi{i}\frac{(2N)!}{(2N+r)!}B_{r,r+2N}(z\mid {\boldsymbol{\omega}})\nonumber \\
{} & \quad
+\frac{(-1)^{N}(2N)!}{(2\pi)^{2N}}
\sum _{k=1}^{r}\omega_{k}^{2N}\sum _{n=1}^{\infty }
\frac{e^{2{\pi}inz_{k}}}{n^{2N+1}}\!\!\prod _{j=1,\,j\neq k}^{r}\!\!(1-e^{2{\pi}in\omega_{jk}})^{-1}.\nonumber
\end{align}
On the other hand, since $z\in D$, we see that $e^{-{\pi}i}(\vert {\boldsymbol{\omega}}\vert^{+}+e^{-{\pi}i}z)=z+e^{-{\pi}i}\vert {\boldsymbol{\omega}}\vert^{+}$.
Therefore, 
\begin{equation}
f(s,z\mid {\boldsymbol{\omega}})
=\zeta_{r}(s,z\mid {\boldsymbol{\omega}})
+(-1)^{r-1}e^{-{\pi}is}\zeta_{r}(s,z+e^{-{\pi}i}\vert {\boldsymbol{\omega}}\vert^{+}\mid e^{-{\pi}i}{\boldsymbol{\omega}}).\nonumber
\end{equation}
Hence, by (\ref{eq:BarnesSV}), Lemma \ref{prop:multiple Brnoulli1} and (\ref{eq:f_{+} D1}),
\begin{align}
\frac{\partial f}{\partial s}(-2N,z\mid {\boldsymbol{\omega}})
&=-\pi{i}(-1)^{r-1}\zeta_{r}(-2N,z+e^{-{\pi}i}\vert {\boldsymbol{\omega}}\vert^{+}\mid e^{-{\pi}i}{\boldsymbol{\omega}})
+\frac{\partial f_{-}}{\partial s}(-2N,z\mid {\boldsymbol{\omega}})\nonumber \\
&=(-1)^{r}\pi{i}\frac{(2N)!}{(2N+r)!}B_{r,r+2N}(z\mid {\boldsymbol{\omega}})\nonumber \\
{} & \quad
+\frac{(-1)^{N}(2N)!}{(2\pi)^{2N}}
\sum _{k=1}^{r}\omega_{k}^{2N}\sum _{n=1}^{\infty }
\frac{e^{-2{\pi}inz_{k}}}{n^{2N+1}}\!\!\prod _{j=1,\,j\neq k}^{r}\!\!(1-e^{-2{\pi}in\omega_{jk}})^{-1}.\nonumber
\end{align}
Consequently we obtain (\ref{eq:Narukawa basic1}). 
Similarly,
\begin{align}
\exp\left(-\frac{\partial f}{\partial s}(0,z\mid {\boldsymbol{\omega}})\right)
&=\exp{\left(\frac{(-1)^{r}\pi{i}}{r!}B_{r,r}(z\mid {\boldsymbol{\omega}})\right)}
\exp{\left(-\frac{\partial f_{+}}{\partial s}(0,z\mid {\boldsymbol{\omega}})\right)}\nonumber \\
&=\exp{\left(\frac{(-1)^{r-1}\pi{i}}{r!}B_{r,r}(z\mid {\boldsymbol{\omega}})\right)}
\exp{\left(-\frac{\partial f_{-}}{\partial s}(0,z\mid {\boldsymbol{\omega}})\right)}.\nonumber
\end{align}
Hence the result (\ref{eq:Narukawa basic2}) follows from (\ref{eq:f_{+} D2}) by analytic continuation.
\end{proof}
\begin{chu}
{\rm(1)} 
For $r=1$, we have
\begin{align}
\frac{\partial f}{\partial s}(-2N,z\mid \omega_{1})
&=\frac{\pi{i}}{2N+1}B_{1,2N+1}(z\mid \omega_{1})
+\frac{(-1)^{N}(2N)!}{(2\pi)^{2N}}
\omega_{1}^{2N}\sum _{n=1}^{\infty }
\frac{e^{2{\pi}ina}}{n^{2N+1}}\\
&=-\frac{\pi{i}}{2N+1}B_{1,2N+1}(z\mid \omega_{1})
+\frac{(-1)^{N}(2N)!}{(2\pi)^{2N}}
\omega_{1}^{2N}\sum _{n=1}^{\infty }
\frac{e^{-2{\pi}ina}}{n^{2N+1}}, \nonumber 
\end{align}
where $z=a\omega_{1}$($0<a<1$). 
In particular, for all $z\in \mathbf{C}$, 
\begin{align}
\exp\left(-\frac{\partial f}{\partial s}(0,z\mid {\boldsymbol{\omega}})\right)
=e^{-\pi{i}\left(\frac{z}{\omega_{1}}-\frac{1}{2}\right)}
(1-e^{2{\pi}ia})
=e^{\pi{i}\left(\frac{z}{\omega_{1}}-\frac{1}{2}\right)}
(1-e^{-2{\pi}ia}).
\end{align}
{\rm(2)} 
We remark that Narukawa has proved (\ref{eq:Narukawa basic2}) in \cite{Na} 
and the left hand side of (\ref{eq:Narukawa basic2}) is a multiple sine function. \\
{\rm(3)} 
If $r=2$, (\ref{eq:Narukawa basic2}) gives Iseki's formula \cite{I}.\\
{\rm(4)} 
By Remark \ref{chu:ext 3}, if $z\in D$, the second equality of (\ref{eq:Narukawa basic1}) and that of (\ref{eq:Narukawa basic2}) are 
true for replacing the condition {\bf[{ORC}]} by $\arg(\omega_{j})\neq \arg(\omega_{k})$ $(j\neq k)$.
\end{chu}

\section*{Acknowledgment}
The author would like to express his gratitude to Professor Masato Wakayama for his mathematical supports, careful reading of this paper and warm encouragements. The author also thanks Professor Katsuhiko Kikuchi for his mathematical advices and warm encouragements, 
and Professors Atsushi Narukawa, Kohji Matsumoto for giving me helpful comments.

\bibliographystyle{amsplain}

\noindent Graduate School of Mathematics, Kyushu University\\
744, Motooka, Nishi-ku, Fukuoka, 819-0395, JAPAN.\\
E-mail: g-shibukawa@math.kyushu-u.ac.jp

\end{document}